%% file: trb-fib.tex
\documentclass[oneside,article]{memoir}

\usepackage{amsmath}         % align-like environments
\usepackage{amsthm,thmtools} % theorem-like environments
\usepackage{enumitem}        % enumerate-like environments

\usepackage{tikz}
  \usetikzlibrary{matrix,arrows}

\usepackage[nobysame,alphabetic,initials]{amsrefs} % bibliography tools
\usepackage[hidelinks]{hyperref}                   % hyperlinks
\usepackage[capitalize,nosort]{cleveref}           % named references

\usepackage{amssymb}
\usepackage[mathscr]{euscript}
\usepackage{relsize}
\usepackage{stmaryrd}
\usepackage{stackengine}

\usepackage{indentfirst} % indent the first paragraph

\isopage[12]

\setlrmargins{*}{*}{1}
\checkandfixthelayout

\counterwithout{section}{chapter}

\input{defs}

\author{Krzysztof Kapulkin \and Karol Szumi{\l}o}
\title{Internal Languages of Finitely Complete $(\infty, 1)$-categories}
\date{\today}

\begin{document}

  \maketitle

  % {\centering \textit{To the memory of Vladimir Voevodsky} \par}

  \begin{abstract}
    We prove that the homotopy theory of Joyal's tribes is equivalent to
    that of fibration categories.
    As a consequence, we deduce a variant of the conjecture asserting that
    Martin-L\"of Type Theory with dependent sums and intensional identity types
    is the internal language of $(\infty, 1)$-categories with finite limits.%
    \MSC{18G55, 55U35, 03B15 (primary)}
  \end{abstract}

  \setlist[enumerate]{label=(\arabic*)}

\input{1-introduction}

\input{2-fibration-categories-and-tribes}

\input{3-semisimplicial-fibration-categories-and-tribes}

\input{4-fibration-categories-of-fibration-categories-and-tribes}

\input{5-presheaves-over-simplicial-categories}

\input{6-hammock-localization-of-a-fibration-category}

\input{7-approximation-fibration-categories-by-tribes}

\input{8-equivalence-between-tribes-and-fibration-categories}

\input{9-application-to-internal-languages}

\input{references}
\end{document}

%% file: defs.tex
  \newcommand{\adj}{\dashv}

  \newcommand{\iso}{\mathrel{\cong}}

  \newcommand{\we}{\mathrel\sim}

  \newcommand{\ev}{\operatorname{ev}}

  \newcommand{\Ho}{\operatorname{Ho}}

  \newcommand{\hamm}{\operatorname{L^H}}

  \makeatletter
  \def\hammrep#1{\expandafter\hammrep@i#1,,\@nil}
  \def\hammrep@i#1,#2,#3\@nil{\hamm #1 (\uvar, #2)}
  \makeatother

  \newcommand{\Frac}{\operatorname{Frac}}

  \newcommand{\fr}{\operatorname{Fr}}
  \newcommand{\frh}{\hat{\fr}}

  \newcommand{\op}{{\mathord\mathrm{op}}}

  \newcommand{\ob}{\operatorname{ob}}

  \newcommand{\pull}{\times}

  \newcommand{\push}{\sqcup}

  \newcommand{\Lan}{\operatorname{Lan}}

  \newcommand{\coend}{\int}

  \newcommand{\slice}{\mathbin\downarrow}

  \newcommand{\fslice}{\mathbin\twoheaddownarrow}

  \newcommand{\Ex}{\operatorname{Ex}}

  \newcommand{\nerve}{\operatorname{N}}

  \newcommand{\Sd}{\operatorname{Sd}}

  \newcommand{\Sk}{\operatorname{Sk}}

  \newcommand{\bd}{\partial}

  \newcommand{\gprod}{\boxbslash}

  \newcommand{\id}{\operatorname{id}}
  
  \newcommand{\cod}{\operatorname{cod}}

  \newcommand{\simp}[1]{\mathord\Delta[#1]}

  \newcommand{\bdsimp}[1]{\mathord{\bd\simp{#1}}}

  \newcommand{\ssimp}[1]{\mathord\Delta_\sharp[#1]}

  \newcommand{\bdssimp}[1]{\bd \Delta_\sharp[#1]}

  \newcommand{\Reedy}{\mathrm{R}}

  \newcommand{\uvar}{\mathord{\relbar}}

  \renewcommand{\tilde}{\widetilde}

  \renewcommand{\hat}{\widehat}

  \renewcommand{\bar}{\widebar}

  \newcommand{\Id}{\mathsf{Id}}

  \newcommand{\tSigma}{\mathsf{\Sigma}}

  \newcommand{\nat}{{\mathord\mathbb{N}}}

  \renewcommand{\emptyset}{\varnothing}

  \newcommand{\cat}[1]{\mathscr{#1}}

  \newcommand{\ccat}[1]{\mathsf{#1}}

  \newcommand{\ncat}[1]{\mathsf{#1}}

  \newcommand{\FibCat}{\ncat{FibCat}}

  \newcommand{\sFibCat}{\ncat{sFibCat}}

  \newcommand{\Trb}{\ncat{Trb}}

  \newcommand{\sTrb}{\ncat{sTrb}}

  \newcommand{\CompCat}{\ncat{CompCat}}

  \newcommand{\Lex}{\ncat{Lex}}

  \newcommand{\Simp}{\Delta}

  \newcommand{\sSimp}{\Delta_\sharp}

  \newcommand{\from}{\colon}

  \newcommand{\ito}{\hookrightarrow}

  \newcommand{\weto}{\mathrel{\ensurestackMath{\stackon[-2pt]{\xrightarrow{\makebox[.8em]{}}}{\mathsmaller{\mathsmaller\we}}}}}
  \newcommand{\weot}{\mathrel{\ensurestackMath{\stackon[-2pt]{\xleftarrow{\makebox[.8em]{}}}{\mathsmaller{\mathsmaller\we}}}}}

  \newcommand{\cto}{\rightarrowtail}

  \newcommand{\acto}{\mathrel{\ensurestackMath{\stackon[-1pt]{\cto}{\mathsmaller{\mathsmaller\we}}}}}

  \newcommand{\fto}{\twoheadrightarrow}

  \newcommand{\afto}{\mathrel{\ensurestackMath{\stackon[-1pt]{\fto}{\mathsmaller{\mathsmaller\we}}}}}

  \newcommand{\rfto}{\mathrel{\rotatebox[origin=c]{-45}{$\Rsh$}}}

  \newcommand{\zto}{\rightsquigarrow}

  \renewcommand{\iff}{if and only if}
  \newcommand{\st}{such that}
  \newcommand{\wrt}{with respect to}
  \newcommand{\tfae}{the following are equivalent}

  \newcommand{\llp}{left lifting property}
  \newcommand{\whe}{weak homotopy equivalence}
  \newcommand{\Whe}{Weak homotopy equivalence}

  \declaretheorem[style=definition,within=section]{definition}
  
  \declaretheorem[style=definition,numberlike=definition]{remark}

  \declaretheorem[style=plain,numberlike=definition]{corollary}
  \declaretheorem[style=plain,numberlike=definition]{lemma}
  \declaretheorem[style=plain,numberlike=definition]{proposition}
  \declaretheorem[style=plain,numberlike=definition]{theorem}

  \declaretheorem[style=plain,numbered=no,name=Theorem]{theorem*}

  \Crefname{corollary}{Corollary}{Corollaries}
  \Crefname{definition}{Definition}{Definitions}
  \Crefname{lemma}{Lemma}{Lemmas}
  \Crefname{proposition}{Proposition}{Propositions}
  \Crefname{remark}{Remark}{Remarks}
  \Crefname{theorem}{Theorem}{Theorems}

  \newcommand{\axm}[1]{(#1\arabic*)}
  \newlist{axioms}{enumerate}{1}
  \Crefname{axiomsi}{}{}

  \newenvironment{tikzeq*}
  {
    \begingroup
    \begin{equation*}
    \begin{tikzpicture}[baseline=(current bounding box.center)]
  }
  {
    \end{tikzpicture}
    \end{equation*}
    \endgroup
    \ignorespacesafterend
  }

  \tikzset
  {
    diagram/.style=
    {
      matrix of math nodes,
      column sep={4.3em,between origins},
      row sep={4em,between origins},
      text height=1.5ex,
      text depth=.25ex
    },
    over/.style={preaction={draw=white,-,line width=6pt}},
    every to/.style={font=\footnotesize},
    inj/.style={right hook->},
    surj/.style={-{Latex[open]}},
    cof/.style={>->},
    fib/.style={->>},
  }

  \DeclareFontFamily{U}{mathx}{\hyphenchar\font45}

  \DeclareFontShape{U}{mathx}{m}{n}{
    <5> <6> <7> <8> <9> <10>
    <10.95> <12> <14.4> <17.28> <20.74> <24.88>
    mathx10}{}

  \DeclareSymbolFont{mathx}{U}{mathx}{m}{n}

  \DeclareFontFamily{U}{mathb}{\hyphenchar\font45}

  \DeclareFontShape{U}{mathb}{m}{n}{
    <5> <6> <7> <8> <9> <10>
    <10.95> <12> <14.4> <17.28> <20.74> <24.88>
    mathb10}{}

  \DeclareSymbolFont{mathb}{U}{mathb}{m}{n}

  \DeclareMathAccent{\widebar}{0}{mathx}{"73}

  \DeclareMathSymbol{\Rsh}{\mathrel}{mathb}{"E9}

  \DeclareFontFamily{U}{MnSymbolA}{}

  \DeclareFontShape{U}{MnSymbolA}{m}{n}{
    <-6> MnSymbolA5
    <6-7> MnSymbolA6
    <7-8> MnSymbolA7
    <8-9> MnSymbolA8
    <9-10> MnSymbolA9
    <10-12> MnSymbolA10
    <12-> MnSymbolA12}{}

  \DeclareSymbolFont{MnSyA}{U}{MnSymbolA}{m}{n}

  \DeclareMathSymbol{\twoheaddownarrow}{\mathrel}{MnSyA}{27}

  \newcommand{\MSC}[1]{%
    \let\thempfn\relax
    \footnotetext[0]{2010 Mathematics Subject Classification: #1.}
  }

%% file: 1-introduction.tex
\section{Introduction}

In recent years, two new frameworks for abstract homotopy theory have emerged:
higher category theory, 
developed extensively by Joyal \cite{Joyal-theory} and Lurie \cite{HTT}, and
Homotopy Type Theory extensively developed in \cite{uf} (and
formalized in \cite{um} and \cite{hott}), referred to as HoTT below.
% Homotopy Type Theory (HoTT) initiated by Voevodsky \cite{um} and
% further developed in \cite{uf}.
The homotopy-theoretic theorems proven in the latter are often labeled as
Synthetic Homotopy Theory, which is supposed to express two ideas.
First, we can reason about objects of an abstract $(\infty,1)$-category as if 
they were spaces. 
Second, a theorem proven in HoTT becomes true in a wide class of higher categories.

% Although the connection between higher categories and HoTT seems intuitive to
% those familiar with both settings, a formal statement of equivalence between them
% was only conjectured in \cites{kl,k} in three variants depending on
% the choice of type constructors and
% the corresponding higher categorical structures.
% These conjectures have far reaching consequences. As mentioned above,
% they allow one to use HoTT to reason about
% sufficiently structured higher categories, for example,
% since the Blakers--Massey Theorem has been proven in HoTT \cite{ffll},
% it is satisfied in an arbitrary $(\infty,1)$-topos.
% Conversely, a type-theoretic statement true in every $(\infty, 1)$-topos must be
% provable in HoTT and consequently, results in higher category theory can suggest
% new principles of logic, such as the Univalence Axiom of Voevodsky.

Although the connection between higher categories and HoTT seems intuitive to
those familiar with both settings, a formal statement of equivalence between them
was only conjectured in \cites{kl,k} in three variants depending on
the choice of type constructors and
the corresponding higher categorical structures.
Slightly more precisely, \cite{kl} provides a link between contextual categories,
an algebraic notion of a model of type theory,
and quasicategories with the appropriate extra structure.
Thus, along with the Initiality Conjecture%
\footnote{\url{https://www.math.ias.edu/vladimir/sites/math.ias.edu.vladimir/files/2015_06_30_RDP.pdf}},
\cite{kl} suggests a relation between syntactically presented type theories
and $(\infty,1)$-categories.

These conjectures have far reaching consequences.
As mentioned above, they allow one to use HoTT to reason about
sufficiently structured higher categories, for example,
since the Blakers--Massey Theorem has been proven in HoTT \cite{ffll},
it is satisfied in an arbitrary $(\infty, 1)$-topos admitting
the relevant Higher Inductive Types.
Conversely, a type-theoretic statement true in every $(\infty, 1)$-topos
must be provable in HoTT and consequently,
results in higher category theory can suggest new principles of logic,
such as the Univalence Axiom of Voevodsky.
% Here, $(\infty, 1)$-topos referes to an \emph{elementary $(\infty, 1)$-topos},
% a notion which is yet to be fully formalized, see \cite{r} for recent progress.
Here, we use the term $(\infty, 1)$-topos for
some yet-to-be-defined notion of \emph{elementary} $(\infty, 1)$-topos.
Although there is no universally agreed upon definition,
a significant progress towards it has been made by Rasekh \cite{r} and
in the unpublished work of Shulman.

In the present paper, we prove a version of the first of these conjectures,
asserting that Martin-L\"of Type Theory with dependent sums and
intensional identity types is the internal language of
finitely complete $(\infty,1)$-categories.
To make this result precise, we assemble the categorical models of type theory,
given by comprehension categories \cite{j},
into a category $\CompCat_{\Id,\Sigma}$ and compare it with
the category $\Lex_\infty$ of quasicategories with finite limits.
Our main theorem (cf.\ \Cref{main-theorem}) states:
\begin{theorem*}
  % The homotopy theory of categorical models of Martin-L\"of Type Theory
  % with dependent sums and intensional identity types is equivalent to
  % the homotopy theory of finitely complete $(\infty, 1)$-categories.
  The $(\infty, 1)$-category of categorical models of Martin-L\"of Type Theory
  with dependent sums and intensional identity types is equivalent to
  the $(\infty, 1)$-category of finitely complete $(\infty, 1)$-categories.
\end{theorem*}
% Ideally, one would like to establish an equivalence between
% suitable syntactically presented type theories and
% finitely complete $(\infty, 1)$-categories and
% the theorem above is an important step in this direction.
% However, a complete result along these lines would require a proof of
% the Initiality Conjecture%
% \footnote{\url{https://www.math.ias.edu/vladimir/sites/math.ias.edu.vladimir/files/2015_06_30_RDP.pdf}}
% of Voevodsky and
% a comparison between contextual categories and comprehension categories,
% both currently being investigated:
% the former by Voevodsky, and the latter by Cho, Knapp, Newstead and Wong.
As indicated above, ideally, one would like to establish
an equivalence between suitable syntactically presented type theories and
finitely complete $(\infty, 1)$-categories and the theorem above is
an important step in this direction.
However, a complete result along these lines would require
a proof of the Initiality Conjecture and
a comparison between contextual categories and comprehension categories
and indeed, both are currently being investigated.

Our approach builds on two recent developments.
First, it was shown in \cite{s3} that
the homotopical category $\Lex_\infty$ is DK-equivalent to that of
fibration categories.
Second, the connection between fibration categories and type theory
was observed in \cite{akl} and then
explored in detail by Shulman \cite{sh} and in the unpublished work of Joyal.
Both Shulman and Joyal gave a categorical axiomatization of
the properties of fibration categories arising from type theory,
introducing the notions of a type-theoretic fibration category and a tribe,
respectively.
The resulting characterizations closely resemble
the \emph{identity type categories} of van den Berg and Garner
\cite{bg}*{Def.~3.2.1}.
Tribes can be seen as a more structured version of fibration categories in that
they are equipped with two clasess of morphisms,
fibrations and anodyne morphisms, that
nearly\footnote{However, fibrations are not necessarily closed under retracts.}
form a weak factorization system with the ``Frobenius property''.
% In particular, tribes are closely related to comprehension categories with
% dependent sums and intensional identity types, but their presentation is
% more convenient for comparison with other settings of abstract homotopy theory.

The equivalence between comprehension categories and tribes is
fairly straightforward and thus the heart of the paper is the proof that
the forgetful functor $\Trb \to \FibCat$
from the homotopical category of tribes to
the homotopical category of fibration categories is a DK-equivalence.
A direct way of accomplishing that would be to construct its homotopy inverse.
However, associating a tribe to a fibration category in
a strictly functorial manner has proven difficult.
Another approach would be to verify Waldhausen's approximation properties which
requires constructing a fibration category structure on $\Trb$.
While this idea does not appear to work directly,
it can refined using semisimplicial methods.

To this end, in course of the proof,
we will consider the following homotopical categories.
\begin{tikzeq*}
\matrix[diagram,column sep={6em,between origins}]
{
                              &[1em] |(sT)| \sTrb & |(sF)| \sFibCat &                   \\
  |(C)| \CompCat_{\Id,\Sigma} &      |(T)|  \Trb  & |(F)|  \FibCat  & |(Q)| \Lex_\infty \\
};

\draw[->] (C) to (T);
\draw[->] (T) to (F);
\draw[->] (F) to (Q);

\draw[->] (sT) to (sF);

\draw[->] (sT) to (T);
\draw[->] (sF) to (F);
\end{tikzeq*}
In the top row, $\sTrb$ and $\sFibCat$ denote
the categories of semisimplicially enriched tribes and
semisimplicially enriched fibration categories, respectively.
The vertical forgetful functors are
directly verified to be homotopy equivalences.
Moreover, both $\sTrb$ and $\sFibCat$ are fibration categories allowing us to
verify that the top map is a DK-equivalence by
checking the approximation properties.

\subsection*{Outline}

In \Cref{fibcat-tribe}, we review background on fibration categories and tribes,
and in \Cref{semisimp-fibcat-tribe}, we introduce
their semisimplicially enriched counterparts.
Then in \Cref{fibcat-of-fibcats}, we construct fibration categories of
semisimplicial fibration categories and semisimplicial tribes,
following \cite{s1}.

To associate a tribe to a fibration category,
we use injective model structures on categories of simplicial presheaves,
which we briefly recall in \Cref{presheaves}.
In \Cref{hammocks}, we study the hammock localization of a fibration category and
construct tribes of representable presheaves over such localizations.

In \Cref{tribes-of-presheaves}, we use them to
verify the approximation properties for
the forgetful functor $\Trb \to \FibCat$.
As mentioned above, this argument is insufficient and we rectify it
in \Cref{equivalence} using semisimplicial enrichments.
Finally in \Cref{language}, we complete the proof by
relating tribes to comprehension categories.

\subsection*{Acknowledgments}

This work would not have been possible without
the support and encouragement of Andr\'e Joyal.
We are also very grateful to him for
sharing an early draft of his manuscript on the theory of tribes with us.

%% file: 2-fibration-categories-and-tribes.tex
\section{Fibration categories and tribes}
\label{fibcat-tribe}

To start off, we discuss the basic theory of fibration categories and tribes.
Fibration categories were first introduced by Brown \cite{br} as
an abstract framework for homotopy theory
alternative to Quillen's model categories.
Our definition differs slightly from Brown's original one in that
we assume the 2-out-of-6 property instead of 2-out-of-3.

\begin{definition}
  A \emph{fibration category} is a category $\cat{C}$ equipped with
  a subcategory of \emph{weak equivalences} (denoted by $\weto$) and
  a subcategory of \emph{fibrations} (denoted by $\fto$) subject to
  the following axioms.
  \begin{axioms}[label=\axm{F}]
  \item\label{fibcat-terminal}
    $\cat{C}$ has a terminal object $1$ and all objects $X$ are fibrant
    (i.e., the morphism $X \to 1$ is a fibration).
  \item\label{fibcat-pullback}
    Pullbacks along fibrations exist in $\cat{C}$ and
    (acyclic) fibrations are stable under pullback.
    (Here, an \emph{acyclic fibration} is a morphism that
    is both a fibration and a weak equivalence.)
  \item\label{fibcat-factorization}
    Every morphism factors as a weak equivalence followed by a fibration.
  \item\label{fibcat-2-out-of-6}
    Weak equivalences satisfy the 2-out-of-6 property.
  \end{axioms}
\end{definition}

We will need a few fundamental facts about fibration categories which
we now recall.
For a more thorough discussion, see \cite{rb}.

\begin{definition}
  Let $\cat{C}$ be a fibration category.
  \begin{enumerate}
  \item A \emph{path object} for an object $a \in \cat{C}$ is
    a factorization of the diagonal morphism $a \to a \times a$ as
    $(\pi_0, \pi_1) \sigma \from a \weto P a \fto a \times a$.
  \item A \emph{homotopy} between morphisms $f, g \from a \to b$ is
    a morphism $H \from a \to P b$ \st{} $(\pi_0, \pi_1) H = (f, g)$.
  \item A morphism $f \from a \to b$ is a \emph{homotopy equivalence} if
    there is a morphism $g \from b \to a$ \st{}
    $g f$ is homotopic to $\id_a$ and $f g$ is homotopic to $\id_b$.
  \item An object $a$ is \emph{cofibrant} if
    for every acyclic fibration $p \from x \afto y$
    there is a lift in every diagram of the form
    \begin{tikzeq*}
    \matrix[diagram,column sep={5em,between origins}]
    {
              & |(x)| x          \\
      |(a)| a & |(y)| y \text{.} \\
    };

    \draw[fib] (x) to node[left] {$\we$} node[right] {$p$} (y);

    \draw[->]        (a) to (y);
    \draw[->,dashed] (a) to (x);
    \end{tikzeq*}
  \end{enumerate}
\end{definition}

\begin{lemma}\label{we-he}
  In a fibration category $\cat{C}$,
  a morphism $f \from a \to b$ between cofibrant objects is a weak equivalence
  \iff{} it is a homotopy equivalence.
\end{lemma}

\begin{proof}
  If $f$ is a weak equivalence, then it is a homotopy equivalence by
  \cite{b}*{Cor.\ 2.12}.

  Conversely, let $f$ be a homotopy equivalence and
  $g \from b \to a$ its homotopy inverse.
  Moreover, let $H$ be a homotopy between $g f$ and $\id_a$ and
  let $G$ be a homotopy between $f g$ and $\id_b$ which
  yield the following diagram.
  \begin{tikzeq*}
  \matrix[diagram]
  {
             &            &          & |(b)| b  &            & |(b0)| b \\
    |(a0)| a &            & |(a)| a  &          & |(Pb)| P b &          \\
             & |(Pa)| P a &          & |(b1)| b &            &          \\
             &            & |(a1)| a &          &            &          \\
  };

  \draw[->] (b)  to node[above left]  {$g$} (a);
  \draw[->] (a)  to node[above right] {$f$} (b1);
  \draw[->] (b1) to node[below right] {$g$} (a1);

  \draw[->] (Pa) to node[below left] {$\we$} (a0);
  \draw[->] (Pa) to node[below left] {$\we$} (a1);

  \draw[->] (a) to node[above]      {$\id$} (a0);
  \draw[->] (a) to node[above left] {$H$}   (Pa);
  \draw[->] (a) to node[left]       {$g f$} (a1);

  \draw[->] (Pb) to node[below right] {$\we$} (b0);
  \draw[->] (Pb) to node[below right] {$\we$} (b1);

  \draw[->] (b) to node[above]       {$\id$} (b0);
  \draw[->] (b) to node[above right] {$G$}   (Pb);
  \draw[->] (b) to node[right]       {$f g$} (b1);
  \end{tikzeq*}
  By 2-out-of-6 it follows that $f$ is a weak equivalence.
\end{proof}

\begin{definition}\label{hpb}
  A commutative square
  \begin{tikzeq*}
  \matrix[diagram]
  {
    |(u)| a & |(x)| c \\
    |(v)| b & |(y)| d \\
  };

  \draw[->] (u) to (v);
  \draw[->] (x) to (y);
  \draw[->] (u) to (x);
  \draw[->] (v) to (y);
  \end{tikzeq*}
  in a fibration category is a \emph{homotopy pullback} if
  given a factorization of $c \weto c' \fto d$,
  the induced morphism $a \to b \pull_d c'$ is a weak equivalence.
\end{definition}

\begin{lemma}\label{hpb-invariant}
  Let $\cat{C}$ be a fibration category.
  \begin{enumerate}
  \item If two squares in $\cat{C}$ are weakly equivalent, then
    one is a homotopy pullback \iff{} the other one is.
  \item Every homotopy pullback in $\cat{C}$ is weakly equivalent to
    a strict pullback along a fibration.
  \item Every homotopy pullback in $\cat{C}$ is weakly equivalent to
    a strict pullback of two fibrations.
  \end{enumerate}
\end{lemma}

\begin{proof}
  These are all well-known;
  they follow directly from the Gluing Lemma \cite{rb}*{Lem.\ 1.4.1(2b)}.
\end{proof}

\begin{definition}
  \leavevmode
  \begin{enumerate}
  \item A \emph{homotopical category} is a category equipped with a class
    of weak equivalences that satisfies the 2-out-of-6 property.
  \item A \emph{homotopical functor} between homotopical categories is
    a functor that preserves weak equivalences.
  \item A \emph{Dwyer--Kan equivalence} (or \emph{DK-equivalence} for short)
    is a homotopical functor that induces an equivalence of homotopy categories
    and \whe{}s on mapping spaces in the hammock localizations,
    cf.\ \cite{dk2}.
  \item A \emph{homotopy equivalence} is
    a homotopical functor $F \from \cat{C} \to \cat{D}$ \st{}
    there is a homotopical functor $G \from \cat{D} \to \cat{C}$ \st{}
    both composites $G F$ and $F G$ are weakly equivalent
    (via zig-zags of natural weak equivalences) to the identity functors.
    A homotopy equivalence is a DK-equivalence,
    cf.\ \cite{dk2}*{Proposition 3.5}.
  \end{enumerate}
\end{definition}

Here, the homotopy category of a homotopical category $\cat{C}$ is
its localization at the class of weak equivalences denoted by $\Ho \cat{C}$.

\begin{definition}
  \leavevmode
  \begin{enumerate}
  \item A functor between fibration categories is \emph{exact} if
    it preserves weak equivalences, fibrations, terminal objects and
    pullbacks along fibrations.
  \item A \emph{weak equivalence} between fibration categories is
     an exact functor that induces an equivalence of the homotopy categories.
  \end{enumerate}
\end{definition}

The homotopical category of small fibration categories will be
denoted by $\FibCat$.

A useful criterion for an exact functor to be a weak equivalence is given by
the following approximation properties.
They were originally introduced by Waldhausen \cite{wa} in the context of
algebraic K-theory and later adapted by Cisinski to
the setting of abstract homotopy theory.
\begin{definition}
  An exact functor $F \from \cat{C} \to \cat{D}$ satisfies
  the \emph{approximation properties} if:
  \begin{enumerate}
  \item[(App1)] $F$ reflects weak equivalences;
  \item[(App2)] for every pair of objects $b \in \cat{C}$, $x \in \cat{D}$ and
    a morphism $x \to F b$, there is a commutative square
    \begin{tikzeq*}
    \matrix[diagram]
    {
      |(x)|  x  & |(b)| F b \\
      |(x')| x' & |(a)| F a \\
    };

    \draw[->] (x') to node[left]  {$\we$} (x);
    \draw[->] (x') to node[below] {$\we$} (a);

    \draw[->] (x) to (b);
    \draw[->] (a) to (b);
    \end{tikzeq*}
    where the labeled morphisms are weak equivalences and
    the one on the right is the image of a morphism $a \to b$.
  \end{enumerate}
\end{definition}

\begin{theorem}[Cisinski]\label{App}
  If $F \from \cat{C} \to \cat{D}$ is
  an exact functor between fibration categories, then \tfae{}:
  \begin{enumerate}
  \item $F$ is a weak equivalence;
  \item $F$ satisfies the approximation properties;
  \item $F$ is a DK-equivalence.
  \end{enumerate}
\end{theorem}

\begin{proof}
  The equivalence between conditions (1) and (2) was proven in
  \cite{c-cd}*{Thm.\ 3.19}.
  The equivalence between conditions (1) and (3) was proven in
  \cite{c-ik}*{Thm.\ 3.25}.
\end{proof}

Properties of categories arising from type theory were first axiomatized by
van den Berg and Garner \cite{bg}
under the name \emph{identity type categories}.
Taking a somewhat different approach to this problem,
Shulman \cite{sh} introduced \emph{type-theoretic fibration categories}.
A similar notion of a tribe was introduced by Joyal and
developed extensively in \cite{Joyal-notes}.
In the remainder of this section, we recall basic definitions and results of
the theory of tribes, almost all of which are folklore.

% Properties of fibration categories arising from type theory were axiomatized by
% Shulman \cite{sh} as \emph{type-theoretic fibration categories}.
% A similar notion of a tribe was introduced by Joyal and developed extensively in
% his unpublished manuscript.
% In the remainder of this section, we recall basic definitions and results of
% the theory of tribes, almost all of which are folklore.

\begin{definition}
  A \emph{tribe} is a category $\cat{T}$ equipped with
  a subcategory whose morphisms are called \emph{fibrations}
  subject to the following axioms.
  (A morphism with the \llp{} \wrt{} all fibrations is called \emph{anodyne} and
  denoted by $\acto$.)
  \begin{axioms}[label=\axm{T}]
  \item\label{tribe-terminal}
    $\cat{T}$ has a terminal object $1$ and all objects are fibrant.
  \item\label{tribe-pullback}
    Pullbacks along fibrations exist in $\cat{T}$ and
    fibrations are stable under pullback.
  \item\label{tribe-factorization}
    Every morphism factors as an anodyne morphism followed by a fibration.
  \item\label{tribe-anodyne}
    Anodyne morphisms are stable under pullbacks along fibrations.
  \end{axioms}
\end{definition}

As mentioned in the introduction,
anodyne morphisms and fibrations
do not necessarily form a weak factorization system,
as the latter need not to be closed under retracts.

\begin{definition}
  Let $\cat{T}$ be a tribe.
  \begin{enumerate}
  \item A \emph{path object} for an object $x \in \cat{T}$ is
    a factorization of the diagonal morphism $x \to x \times x$ as
    $(\pi_0, \pi_1) \sigma \from x \acto P x \fto x \times x$.
  \item A \emph{homotopy} between morphisms $f, g \from x \to y$ is
    a morphism $H \from x \to P y$ \st{} $(\pi_0, \pi_1) H = (f, g)$.
  \item A morphism $f \from x \to y$ is a \emph{homotopy equivalence} if
    there is a morphism $g \from y \to x$ \st{}
    $g f$ is homotopic to $\id_x$ and $f g$ is homotopic to $\id_y$.
  \end{enumerate}
\end{definition}

\begin{lemma}[cf.\ \cite{sh}*{Lem.\ 3.6}]\label{anodyne-he}
  An anodyne morphism $f \from x \acto y$ in a tribe is a homotopy equivalence.
\end{lemma}

\begin{proof}
  Using the lifting property of $f$ against $x \fto 1$ we obtain
  a retraction $r \from y \to x$.
  On the other hand, a lift in
  \begin{tikzeq*}
  \matrix[diagram,column sep={6em,between origins}]
  {
    |(x)| x & |(P)|  P y        \\
    |(y)| y & |(yy)| y \times y \\
  };

  \draw[cof] (x) to node[left] {$f$} node[right] {$\we$} (y);

  \draw[->]  (y) to node[below] {$(\id, f r)$}     (yy);
  \draw[->]  (x) to node[above] {$\sigma f$}       (P);
  \draw[fib] (P) to node[right] {$(\pi_0, \pi_1)$} (yy);
  \end{tikzeq*}
  is a homotopy between $\id_y$ and $f r$.
\end{proof}

\begin{lemma}[Joyal]\label{he-2-out-of-6}
  Homotopy equivalences in a tribe $\cat{T}$ are saturated, i.e.,
  a morphism is a homotopy equivalence \iff{}
  it becomes an isomorphism in $\Ho \cat{T}$.
  In particular, homotopy equivalences satisfy 2-out-of-6.
\end{lemma}

\begin{proof}
  Let $\cat{B}$ be an arbitrary category and $F \from \cat{T} \to \cat{B}$ be
  any functor.
  If $F$ identifies homotopic morphisms, then
  it carries homotopy equivalences to isomorphisms by definition.
  Conversely, if $F$ inverts homotopy equivalences, then
  it identifies homotopic morphisms.
  Indeed, this follows from the fact that for every object $x$,
  the morphism $\sigma \from x \to P x$ is a homotopy equivalence
  by \Cref{anodyne-he}.
  Thus the localization of $\cat{T}$ at homotopy equivalences
  coincides with its quotient by the homotopy relation
  which implies saturation.
  Consequently, homotopy equivalences satisfy 2-out-of-6 since
  isomorphisms in $\Ho \cat{T}$ do.
\end{proof}

\begin{lemma}[cf.\ \cite{sh}*{Lem.\ 3.7}]\label{anodyne-cancellation}
  If $f \from x \to y$ and $g \from y \to z$ are morphisms \st{}
  $g$ and $g f$ are anodyne, then so is $f$.
\end{lemma}

\begin{proof}
  Since $g$ is anodyne, there is a lift in the square
  \begin{tikzeq*}
  \matrix[diagram,column sep={5em,between origins}]
  {
    |(y0)| y & |(y1)| y          \\
    |(z)|  z & |(1)|  1 \text{.} \\
  };

  \draw[->] (y0) to node[above] {$\id$} (y1);

  \draw[cof] (y0) to node[left] {$g$} node[right] {$\we$} (z);

  \draw[->]  (z)  to (1);
  \draw[fib] (y1) to (1);

  \draw[->,dashed] (z) to node[above left] {$r$} (y1);
  \end{tikzeq*}
  The diagram
  \begin{tikzeq*}
  \matrix[diagram]
  {
    |(xl)| x & |(xm)| x & |(xr)| x \\
    |(yl)| y & |(z)|  z & |(yr)| y \\
  };

  \draw[->] (xl) to node[left]  {$f$}   (yl);
  \draw[->] (xr) to node[right] {$f$}   (yr);

  \draw[cof] (xm) to node[left] {$g f$} node[right] {$\we$} (z);

  \draw[->] (xl) to node[above] {$\id$} (xm);
  \draw[->] (xm) to node[above] {$\id$} (xr);

  \draw[->] (yl) to node[below] {$g$} (z);
  \draw[->] (z)  to node[below] {$r$} (yr);
  \end{tikzeq*}
  shows that $f$ is a retract of $g f$ and so it is anodyne.
\end{proof}

\begin{lemma}[Joyal]\label{anodyne-pullback}
  Given a commutative diagram of the form
  \begin{tikzeq*}
  \matrix[diagram]
  {
    & |(X0)| x_0 & & |(Y0)| y_0   \\[-1.5em]
      |(X1)| x_1 & & |(Y1)| y_1 & \\
      |(A)|  a   & & |(B)|  b   & \\
  };

  \draw[fib] (X0) to (A);
  \draw[fib] (Y0) to (B);

  \draw[->]  (X0) to (X1);
  \draw[cof] (Y0) to node[above left] {$\we$} (Y1);

  \draw[fib] (X1) to (A);
  \draw[fib] (Y1) to (B);

  \draw[->,over] (X1) to (Y1);
  \draw[->]      (X0) to (Y0);
  \draw[->]      (A)  to (B);
  \end{tikzeq*}
  where all squares are pullbacks,
  if $y_0 \to y_1$ is anodyne, then so is $x_0 \to x_1$.
\end{lemma}

\begin{proof}
  Pick a factorization $a \acto a' \fto b$ and form pullback squares
  \begin{tikzeq*}
    \matrix[diagram]
    {
      |(x0)| x_0 & |(x'0)| x'_0 & |(y0)| y_0        \\
      |(x1)| x_1 & |(x'1)| x'_1 & |(y1)| y_1        \\
      |(A)|  a   & |(A')|  a'   & |(B)|  b \text{.} \\
    };

    \draw[->] (x0) to (x1);
    \draw[->] (x1) to (A);

    \draw[->] (x'0) to (x'1);
    \draw[->] (x'1) to (A');

    \draw[cof] (y0) to node[right] {$\we$} (y1);
    \draw[fib] (y1) to (B);

    \draw[->] (x0)  to (x'0);
    \draw[->] (x'0) to (y0);

    \draw[->] (x1)  to (x'1);
    \draw[->] (x'1) to (y1);

    \draw[cof] (A)  to node[below] {$\we$} (A');
    \draw[fib] (A') to (B);
  \end{tikzeq*}
  The morphism $x'_1 \to a'$ is a fibration
  as a pullback of a fibration $y_1 \to b$ and
  thus the morphism $x_1 \to x'_1$ is anodyne as
  a pullback of an anodyne morphism $a \to a'$ along a fibration.
  Similarly, $x'_0 \to x'_1$ is anodyne.
  Furthermore, since $y_0 \to b$ is also a fibration,
  the same argument implies that $x_0 \to x'_0$ is anodyne.
  Therefore, $x_0 \to x_1$ is also anodyne by \Cref{anodyne-cancellation}.
\end{proof}

\begin{definition}
  \leavevmode
  \begin{enumerate}
  \item A functor between tribes is a \emph{homomorphism} if
    it preserves fibrations, anodyne morphisms, terminal objects and
    pullbacks along fibrations.
  \item A homomorphism between tribes is a \emph{weak equivalence} if it
    induces an equivalence of the homotopy categories.
  \end{enumerate}
\end{definition}

The homotopical category of small tribes will be denoted by $\Trb$.

\begin{theorem}\label{Trb-FibCat}
  Every tribe with its subcategories of fibrations and homotopy equivalences
  is a fibration category.
  Moreover, every homomorphism of tribes is an exact functor.
  This yields a homotopical forgetful functor $\Trb \to \FibCat$.
\end{theorem}

\begin{proof}
  In \cite{sh}*{Thm.\ 3.13}, it is proven that
  every type-theoretic fibration category is
  a ``category of fibrant objects''.
  A type-theoretic fibration category is defined just like a tribe except that
  the statement of \Cref{anodyne-pullback} is used in addition to
  axiom \Cref{tribe-anodyne}.
  (The definition also includes an additional axiom about
  dependent products which is not used in the proof of Thm.\ 3.13.)
  Similarly, a category of fibrant objects is
  defined just like a fibration category except that only 2-out-of-3 is assumed
  in the place of 2-out-of-6.
  However, the latter was verified in \Cref{he-2-out-of-6}.

  A homomorphism of tribes preserves fibrations, terminal objects and
  pullbacks along fibrations by definition.
  It also preserves anodyne morphism and hence path objects and, consequently,
  homotopies and homotopy equivalences.
  Thus it is an exact functor.
\end{proof}

For clarity of exposition, from this point on,
we will refer to homotopy equivalences in a tribe as weak equivalences.

\begin{lemma}[cf.\ \cite{sh}*{Lem.\ 3.11}]\label{tribe-cofibrant}
  Every acyclic fibration in a tribe admits a section.
  In particular, every object in a tribe is cofibrant. \qed
\end{lemma}

\begin{lemma}\label{anodyne-gluing}
  Let
  \begin{tikzeq*}
  \matrix[diagram,column sep={5em,between origins}]
  {
      |(xe)| x_\emptyset & & |(x1)|  x_1    & \\
    & |(ye)| y_\emptyset & & |(y1)|  y_1      \\
      |(x0)| x_0         & & |(x01)| x_{01} & \\
    & |(y0)| y_0         & & |(y01)| y_{01}   \\
  };

  \draw[->]  (xe) to (x0);
  \draw[fib] (x1) to (x01);

  \draw[->] (xe) to (x1);
  \draw[->] (x0) to (x01);

  \draw[->,over] (ye) to (y0);
  \draw[fib]     (y1) to (y01);

  \draw[->,over] (ye) to (y1);
  \draw[->]      (y0) to (y01);

  \draw[->] (xe)  to (ye);

  \draw[cof] (x0)  to node[above right] {$\we$} (y0);
  \draw[cof] (x1)  to node[above right] {$\we$} (y1);
  \draw[cof] (x01) to node[above right] {$\we$} (y01);
  \end{tikzeq*}
  be a cube in a tribe where
  $x_1 \fto x_{01}$ and $y_1 \fto y_{01}$ are fibrations and
  the front and back squares are pullbacks.
  If all $x_0 \acto y_0$, $x_1 \acto y_1$ and $x_{01} \acto y_{01}$ are anodyne,
  then so is $x_\emptyset \to y_\emptyset$.
\end{lemma}

\begin{proof}
  Taking pullbacks in the right and left faces, we obtain a diagram
  \begin{tikzeq*}
  \matrix[diagram,column sep={5em,between origins}]
  {
        |(xe)| x_\emptyset & & |(x1)|  x_1    & & \\
    &   |(z)|  z           & & |(zp)|  z'     &   \\
    & & |(ye)| y_\emptyset & & |(y1)|  y_1        \\
    &   |(x0)| x_0         & & |(x01)| x_{01} &   \\
    & & |(y0)| y_0         & & |(y01)| y_{01}     \\
  };

  \draw[->] (xe) to (x1);
  \draw[->] (xe) to (z);
  \draw[->] (x1) to (zp);

  \draw[->]  (xe) to [bend right=20] (x0);
  \draw[fib] (x1) to [bend right=20] (x01);

  \draw[->] (z)  to (x0);
  \draw[->] (zp) to (x01);

  \draw[->,over] (z)  to (zp);
  \draw[->]      (x0) to (x01);

  \draw[->,over] (ye) to (y0);
  \draw[fib]     (y1) to (y01);

  \draw[->,over] (ye) to (y1);
  \draw[->]      (y0) to (y01);

  \draw[->] (z)  to (ye);
  \draw[->] (zp) to (y1);

  \draw[cof] (x0)  to node[above right] {$\we$} (y0);
  \draw[cof] (x01) to node[above right] {$\we$} (y01);
  \end{tikzeq*}
  where all the downward arrows are fibrations by \Cref{tribe-pullback}.
  Moreover, $z \to y_\emptyset$ and $z' \to y_1$ are anodyne
  by \Cref{tribe-anodyne}.
  Thus $x_1 \to z'$ is anodyne by \Cref{anodyne-cancellation} and
  \Cref{anodyne-pullback} implies that $x_\emptyset \to z$ is anodyne.
  It follows that the composite $x_\emptyset \to z \to y_\emptyset$ is
  also anodyne as required.
\end{proof}

\begin{corollary}[Joyal]\label{anodyne-product}
  The product of anodyne morphisms is anodyne. \qed
\end{corollary}

We conclude this section by constructing fibration categories and tribes of
Reedy fibrant diagrams.
The argument given in the proof of the lemma below is standard,
but it will reappear in various modified forms throughout the paper.

\begin{definition}
  \leavevmode
  \begin{enumerate}
  \item[(1)] A category $J$ is \emph{inverse} if there is a function,
    called \emph{degree}, $\deg \from \ob(J) \to \nat$ \st{}
    for every non-identity map $j \to j'$ in $J$ we have $\deg(j) > \deg(j')$.
  \end{enumerate}
  Let $J$ be an inverse category.
  \begin{enumerate}
  \item[(2)] Let $j \in J$.
    The \emph{matching category} $\bd (j \slice J)$ of $j$ is
    the full subcategory of the slice category $j \slice J$ consisting of
    all objects except $\id_j$.
    There is a canonical functor $\bd (j \slice J) \to J$,
    assigning to a morphism (regarded as an object of $\bd (j \slice J)$)
    its codomain.
  \item[(3)] Let $X \from J \to \cat{C}$ and $j \in J$.
    The \emph{matching object} of $X$ at $j$ is defined as a limit of
    the composite
    \begin{tikzeq*}
    \matrix[diagram]
    {
               |(M)| M_j X := \lim \big( \bd (j \slice J)
      &[3.5em] |(J)| J
      &[-1em]  |(C)| \cat{C} \big) \text{.} \\
    };

    \draw[->] (M) to (J);

    \draw[->] (J) to node[above] {$X$} (C);
    \end{tikzeq*}
    The canonical morphism $X_j \to M_j X$ is called
    the \emph{matching morphism}.
  \item[(4)] Let $\cat{C}$ be a fibration category.
     A diagram $X \from J \to \cat{C}$ is called \emph{Reedy fibrant},
     if for all $j \in J$, the matching object $M_j X$ exists and
     the matching morphism $X_j \to M_j X$ is a fibration.
  \item[(5)] Let $\cat{C}$ be a fibration category and
     let $X, Y \from J \to \cat{C}$ be Reedy fibrant diagrams in $\cat{C}$.
     A morphism $f \from X \to Y$ of diagrams is a \emph{Reedy fibration} if
     for all $j \in J$ the induced morphism $X_j \to M_j X \pull_{M_j Y} Y_j $
     is a fibration.
  \item[(6)] If $\cat{C}$ is a fibration category, then
    $\cat{C}^J_\Reedy$ denotes the category of
    Reedy fibrant diagrams in $\cat{C}$ over $J$.
    If $J$ is a homotopical inverse category
    (i.e., inverse category equipped with a homotopical structure),
    then $\cat{C}^J_\Reedy$ will denote
    the category of Reedy fibrant homotopical diagrams in $\cat{C}$ over $J$.
  \end{enumerate}
\end{definition}

The second part of the following lemma has been proven in \cite{sh}*{Thm.~11.11}
in the special case of plain (non-homotopical) inverse category $J$.
Here, we provide a homotopical generalization.

\begin{lemma}\label{tribe-Reedy}
  Let $J$ be a homotopical inverse category.
  \begin{enumerate}
  \item If $\cat{C}$ is a fibration category, then so is $\cat{C}^J_\Reedy$ with
    levelwise weak equivalences and Reedy fibrations.
  \item If $\cat{T}$ is a tribe, then so is $\cat{T}^J_\Reedy$ with
    Reedy fibrations.
    Moreover, both anodyne morphisms and weak equivalences in $\cat{T}^J_\Reedy$
    are levelwise.
  \end{enumerate}
\end{lemma}

\begin{proof}
  Part (1) is \cite{rb}*{Thm.\ 9.3.8(1a)}.

  In part (2), \Cref{tribe-terminal} is evident and
  \Cref{tribe-pullback} is verified exactly as in part (1).

  Every morphism in $\cat{T}^J_\Reedy$ factors into
  a levelwise anodyne morphism followed by a Reedy fibration by
  the proof of \cite{ho}*{Thm.\ 5.1.3}.
  Moreover, levelwise anodyne morphisms have the \llp{} \wrt{} Reedy fibrations
  by \cite{ho}*{Prop.\ 5.1.4}.
  In particular, they are anodyne.

  Thus for \Cref{tribe-factorization} it suffices to verify that
  every anodyne morphism $x \acto y$ is levelwise anodyne.
  Factor it into a levelwise anodyne morphism $x \to x'$ followed by
  a Reedy fibration $x' \to y$.
  Then there is a lift in
  \begin{tikzeq*}
  \matrix[diagram]
  {
    |(A)|  x & |(A')| x' \\
    |(B0)| y & |(B1)| y  \\
  };

  \draw[->]  (A)  to (A');
  \draw[->]  (B0) to (B1);
  \draw[fib] (A') to (B1);

  \draw[cof] (A) to node[left] {$\we$} (B0);
  \end{tikzeq*}
  which exhibits $x \to y$ as a retract of $x \to x'$ and hence
  the former is levelwise anodyne.
  Since anodyne morphisms coincide with levelwise anodyne morphisms,
  they are stable under pullback along Reedy fibrations
  (which are, in particular, levelwise), which proves \Cref{tribe-anodyne}.

  Let $\cat{T}^J_{\mathrm{lvl}}$ be
  the fibration category of
  Reedy fibrant diagrams in the underlying fibration category of $\cat{T}$ with
  levelwise weak equivalences as constructed in part (1).
  We verify that a morphism $f$ is a weak equivalence in $\cat{T}^J_\Reedy$
  \iff{} it is a weak equivalence in $\cat{T}^J_{\mathrm{lvl}}$.
  A path object in $\cat{T}^J_\Reedy$ is also
  a path object in $\cat{T}^J_{\mathrm{lvl}}$
  since anodyne morphisms in $\cat{T}^J_\Reedy$ are levelwise.
  It follows that $f$ is a weak equivalence in $\cat{T}^J_\Reedy$ \iff{}
  it is a homotopy equivalence in $\cat{T}^J_{\mathrm{lvl}}$.
  By \Cref{we-he}, it suffices to verify that
  all objects of $\cat{T}^J_{\mathrm{lvl}}$ are cofibrant.
  By \cite{ho}*{Prop.\ 5.1.4}, an object of $\cat{T}^J_{\mathrm{lvl}}$ is
  cofibrant \iff{} it is levelwise cofibrant so the conclusion follows
  by \Cref{tribe-cofibrant}.
\end{proof}

%% file: 3-semisimplicial-fibration-categories-and-tribes.tex
\section{Semisimplicial fibration categories and tribes}
\label{semisimp-fibcat-tribe}

As mentioned in the introduction,
the homotopical category $\Trb$ does not appear to carry a structure of
a fibration category for the reasons that
will be explained in \Cref{fibcat-of-fibcats}.
To resolve this issue we introduce semisimplicial enrichments of
both tribes and fibration categories, following Schwede \cite{s}*{Sec.~3}.
The key ingredient of this solution is the notion of a frame.
The category of frames in a fibration category (tribe) carries
a natural structure of a semisimplicial fibration category (tribe).
Although our approach is inspired by Joyal's theory of simplicial tribes,
we were forced to modify the enriching category since
the construction of the category of frames has no simplicial counterpart.

We begin this section by reviewing basics of semisimplicial sets.
(For a more complete account, see \cite{rs}.)
Let $\sSimp$ denote the category whose objects are
finite non-empty totally ordered sets of the form $[m] = \{ 0 < \ldots < m \}$
and morphisms are injective order preserving maps.
A \emph{semisimplicial set} is a presheaf over $\sSimp$.
A representable semisimplicial set will be denoted by $\ssimp{m}$ and
its boundary (obtained by removing the top-dimensional simplex)
by $\bdssimp{m}$.

The inclusion $\sSimp \ito \Simp$ induces a forgetful functor from
simplicial sets to semisimplicial sets which admits a left adjoint.
\emph{\Whe{}s} of semisimplicial sets are created by this adjoint from
\whe{}s of simplicial sets.

The \emph{geometric product} of semisimplicial sets $K$ and $L$ is defined by
the coend formula
\begin{equation*}
  K \gprod L = \coend^{[m], [n]} K_m \times L_n \times \nerve_\sharp ([m] \times [n])
\end{equation*}
where $\nerve_\sharp P$ is the semisimplicial nerve of a poset $P$, i.e.,
the semisimplicial set whose $k$-simplices are
injective order preserving maps $[k] \ito P$.
This defines a symmetric monoidal structure on
the category of semisimplicial sets with $\ssimp{0}$ as a unit.
(This monoidal structure can be seen as an example of
a ``Day convolution'' structure induced by
a promonoidal structure of $\Simp_\sharp$.)

\begin{definition}
  A \emph{semisimplicial fibration category} $\cat{C}$ is a fibration
  category that carries a semisimplicial enrichment with cotensors by
  finite semisimplicial sets satisfying the following \emph{pullback-cotensor property}.
  \begin{enumerate}
  \item[(SF)] If $i \from K \cto L$ is
    a monomorphism between finite semisimplicial sets and
    $p \from a \fto b$ is a fibration in $\cat{C}$, then
    the induced morphism $(i^*, p_*) \from a^L \to a^K \pull_{b^K} b^L$ is
    a fibration. 
    Moreover, if either $i$ or $p$ is acyclic, then so is $(i^*, p_*)$.
  \end{enumerate}
\end{definition}

\begin{definition}
  A \emph{semisimplicial tribe} $\cat{T}$ is a tribe that
  carries a semisimplicial enrichment that
  makes the underlying fibration category semisimplicial and
  satisfies the following condition.
  \begin{enumerate}
  \item[(ST)] If $K$ is a finite semisimplicial set and $x \acto y$ is anodyne,
    then so is $x^K \to y^K$.
  \end{enumerate}
\end{definition}

\begin{lemma}\label{semisimplicial-tribe-Reedy}
  Let $J$ be a homotopical inverse category.
  \begin{enumerate}
  \item If $\cat{C}$ is a semisimplicial fibration category, then
    so is $\cat{C}^J_\Reedy$.
  \item If $\cat{T}$ is a semisimplicial tribe, then so is $\cat{T}^J_\Reedy$.
  \end{enumerate}  
\end{lemma}

\begin{proof}
  This follows from \Cref{tribe-Reedy}, \cite{s}*{Prop.~3.5} and
  the fact that cotensors are levelwise.
\end{proof}

\begin{definition}
  \leavevmode
  \begin{enumerate}
  \item A semisimplicial functor between semisimplicial fibration categories is
    \emph{exact} if it is exact as a functor of
    the underlying fibration categories and
    it preserves cotensors by finite semisimplicial sets.
  \item A semisimplicial functor between semisimplicial tribes is
    a \emph{homomorphism} if it is a homomorphism of the underlying tribes and
    it preserves cotensors by finite semisimplicial sets.
  \end{enumerate}
\end{definition}

\begin{definition}
  \leavevmode
  \begin{enumerate}
  \item A \emph{weak equivalence} of semisimplicial fibration categories is
    an exact functor that is a weak equivalence of
    their underlying fibration categories.
  \item A \emph{weak equivalence} of semisimplicial tribes is
    a homomorphism that is a weak equivalence of their underlying tribes.
  \end{enumerate}
\end{definition}

A \emph{frame} in a fibration category $\cat{C}$ is a Reedy fibrant,
homotopically constant semisimplicial object in $\cat{C}$.
The category of frames in a fibration category $\cat{C}$
will be denoted by $\fr \cat{C}$.
If $\cat{C}$ is semisimplicial, then every object $a$ has
the \emph{canonical frame} $a^{\ssimp{\uvar}}$.

\begin{lemma}\label{anodyne-end}
  If $x \to y$ is a levelwise anodyne extension between frames in a tribe and
  $K$ is a finite semisimplicial set, then
  $\coend_{[n]} x_n^{K_n} \to \coend_{[n]} y_n^{K_n}$ is anodyne
  (in particular, these ends exist).
\end{lemma}

\begin{proof}
  Let $k$ be the dimension of K, i.e., the largest $l$ \st{} $K_l$ is non-empty.
  Let $\Sk^{k - 1} K$ be the $(k - 1)$-skeleton of $K$, i.e.,
  the semisimplicial set consisting of simplices of $K$ of
  dimension at most $k - 1$.
  % The argument proceeds by induction \wrt{} $k$
  % for all finite dimensional semisimplicial sets $K$ simultaneously.
  The end $\coend_{[n]} x_n^{K_n}$ is constructed by induction \wrt{} $k$.
  For $k = -1$ (i.e., $K = \emptyset$) it is just the terminal object.
  For $k \ge 0$, we have a pushout
  \begin{tikzeq*}
  \matrix[diagram,column sep={7em,between origins}]
  {
    |(b)| \bdssimp{k} \times K_k & |(k1)| \Sk^{k-1} K \\
    |(s)|   \ssimp{k} \times K_k & |(k)|            K \\
  };

  \draw[->,shorten >=0.2em] (k1) to (k);

  \draw[inj] (b)  to (s);
  \draw[->]  (b)  to (k1);
  \draw[->]  (s)  to (k);
  \end{tikzeq*}
  which yields a pullback
  \begin{tikzeq*}
  \matrix[diagram,column sep={10em,between origins}]
  {
    |(k)|  \coend_{[n]} x_n^{(\Sk^k K)_n}     & |(s)| \coend_{[n]} x_n^{(\ssimp{k} \times K_k)_n}            \\[1em]
    |(k1)| \coend_{[n]} x_n^{(\Sk^{k-1} K)_n} & |(b)| \coend_{[n]} x_n^{(\bdssimp{k} \times K_k)_n} \text{.} \\
  };

  \draw[fib,shorten >=0.5em] (s)  to (b);
  \draw[->,shorten >=0.5em]  (k)  to (k1);

  \draw[->]  (k1) to (b);
  \draw[->]  (k)  to (s);
  \end{tikzeq*}
  The ends on the right exist since
  they coincide with $x_k^{K_k}$ and $(M_k x)^{K_k}$.
  Moreover, the right arrow coincides with
  the $K_k$-fold power of the matching morphism of $x$ and hence
  it is a fibration.
  The end $\coend_{[n]} x_n^{(\Sk^{k-1} K)_n}$ exists
  by the inductive hypothesis and thus so does $\coend_{[n]} x_n^{(\Sk^k K)_n}$.

  For $k = -1$, the morphism $\coend_{[n]} x_n^{K_n} \to \coend_{[n]} y_n^{K_n}$
  % \begin{tikzeq*}
  % \matrix[diagram,column sep={8em,between origins}]
  % {
  %   |(x)| \coend_{[n]} x_n^{K_n} &
  %   |(y)| \coend_{[n]} y_n^{K_n} \\
  % };
  %
  % \draw[->] (x) to (y);
  % \end{tikzeq*}
  is an isomorphism.
  For $k \ge 0$ we note that
  \begin{tikzeq*}
  \matrix[diagram,column sep={9.5em,between origins}]
  {
    |(x)| \coend_{[n]} x_n^{(\ssimp{k} \times K_k)_n} &
    |(y)| \coend_{[n]} y_n^{(\ssimp{k} \times K_k)_n} \\
  };
  
  \draw[->] (x) to (y);
  \end{tikzeq*}
  coincides with $x_k^{K_k} \to y_k^{K_k}$ so it is anodyne
  by \Cref{anodyne-product}.
  Moreover, the morphisms
  \begin{tikzeq*}
  \matrix[diagram,column sep={9em,between origins}]
  {
    |(x0)| \coend_{[n]} x_n^{(\Sk^{k-1} K)_n}            &
    |(y0)| \coend_{[n]} y_n^{(\Sk^{k-1} K)_n}            &[-3em]
           \text{and}                                    &[-3em]
    |(x1)| \coend_{[n]} x_n^{(\bdssimp{k} \times K_k)_n} &[1em]
    |(y1)| \coend_{[n]} y_n^{(\bdssimp{k} \times K_k)_n} \\
  };
  
  \draw[->] (x0) to (y0);
  \draw[->] (x1) to (y1);
  \end{tikzeq*}
  are anodyne by the inductive hypothesis.
  Thus the assumptions of \Cref{anodyne-gluing} are satisfied when
  we apply it to the map from the square above to the analogous square for $y$.
  It follows that
  \begin{tikzeq*}
  \matrix[diagram,column sep={8em,between origins}]
  {
    |(x)| \coend_{[n]} x_n^{(\Sk^k K)_n} &
    |(y)| \coend_{[n]} y_n^{(\Sk^k K)_n} \\
  };
  
  \draw[->] (x) to (y);
  \end{tikzeq*}
  is also anodyne.
\end{proof}

The following theorem establishes a semisimplicial enrichment on $\fr \cat{C}$.
In this enrichment, the object
\begin{equation*}
  (K \rhd a)_m = \coend_{[n] \in \sSimp} a_n^{(\ssimp{m} \gprod K)_n}
\end{equation*}
is a cotensor of $a$ by $K$ by the proof of \cite{s}*{Thm.\ 3.17}.
Moreover, for frames $a$ and $b$, the hom-object is a semisimplicial set whose
$m$-simplices are maps of frames $a \to \ssimp{m} \rhd b$.

\begin{theorem}\label{tribe-frames}
  \leavevmode
  \begin{enumerate}
  \item For a fibration category $\cat{C}$,
    the category of frames $\fr \cat{C}$ is a semisimplicial fibration category
    and the evaluation at $0$ functor $\fr \cat{C} \to \cat{C}$ is
    a weak equivalence.
  \item For a tribe $\cat{T}$,
    the category of frames $\fr \cat{T}$ is a semisimplicial tribe and
    the evaluation at $0$ functor $\fr \cat{T} \to \cat{T}$ is
    a weak equivalence.
  \end{enumerate}
\end{theorem}

\begin{proof}
  Part (1) is \cite{s}*{Thms.\ 3.10 and 3.17}.

  For part (2), $\fr \cat{T}$ is a tribe by \Cref{tribe-Reedy}.
  The axiom (SF) follows from part (1) and (ST) follows from \Cref{anodyne-end}.
\end{proof}

If $\cat{C}$ is a semisimplicial fibration category, then
\Cref{semisimplicial-tribe-Reedy} yields
another structure of a semisimplicial fibration category on $\fr \cat{C}$ with
levelwise cotensors.
In general, this structure differs from
the one described in \Cref{tribe-frames} and in the remainder of the paper
we will always consider the latter.
However, the two cotensor operations agree on canonical frames.

\begin{lemma}\label{frame-cotensor}
  Let $\cat{C}$ be a semisimplicial fibration category and $a \in \cat{C}$.
  Then for every finite semisimplicial set $K$,
  we have $K \rhd (a^{\ssimp{\uvar}}) \iso (a^K)^{\ssimp{\uvar}}$.
\end{lemma}

\begin{proof}
  We have the following string of isomorphisms, natural in $m$:
  \begin{equation*}
    (K \rhd (a^{\ssimp{\uvar}}))_m
      \iso \coend_{[n]} (a^{\ssimp{n}})^{(\ssimp{m} \gprod K)_n}
      \iso a^{\coend^{[n]} \ssimp{n} \times (\ssimp{m} \gprod K)_n}
      \iso a^{K \gprod \ssimp{m}} \iso (a^K)^{\ssimp{m}} \text{.} \qedhere
  \end{equation*}
\end{proof}

We now need to establish a semisimplicial equivalence between
a semisimplicial fibration category $\cat{C}$ and $\fr \cat{C}$
(note that the evaluation at $0$ functor of \Cref{tribe-frames} is
not semisimplicial).
By the preceding lemma, there is
a semisimplicial exact functor $\cat{C} \to \fr \cat{C}$ given by
$a \mapsto a^{\ssimp{\uvar}}$ which is in fact a weak equivalence.
However, it is only pseudonatural in $\cat{C}$ which
is insufficient for our purposes.
We can correct this defect by introducing a modified version of $\fr \cat{C}$.

Let $\hat{[1]}$ denote the homotopical category with underlying category $[1]$
whose all morphisms are weak equivalences.
Let $\frh \cat{C}$ be the full subcategory of
$(\fr \cat{C})^{\hat{[1]}}_\Reedy$ spanned by objects $X$ \st{}
$X_1$ is (isomorphic to) the canonical frame on $X_{1,0}$.
It is a variant of the gluing construction
(in the sense of \cite{sh}*{Sec.~13}, inspired by Artin gluing)
along the canonical frame functor $\cat{C} \to \fr \cat{C}$.
More precisely, it is its full subcategory spanned by
the objects whose structure maps are weak equivalences.

\begin{proposition}
  \leavevmode
  \begin{enumerate}
  \item If $\cat{C}$ is a semisimplicial fibration category,
  then so is $\frh \cat{C}$.
  Moreover, both the evaluation at $0$ functor $\frh \cat{C} \to \fr \cat{C}$
  and the evaluation at $(1, 0)$ functor $\frh \cat{C} \to \cat{C}$ are
  semisimplicial exact.
  \item If $\cat{T}$ is a semisimplicial tribe, then so is $\frh \cat{T}$.
  Moreover, both the evaluation at $0$ functor $\frh \cat{T} \to \fr \cat{T}$
  and the evaluation at $(1, 0)$ functor $\frh \cat{T} \to \cat{T}$ are
  semisimplicial homomorphisms.
  \end{enumerate}
\end{proposition}

\begin{proof}
  If $\cat{C}$ is a semisimplicial fibration category,
  then so is $(\fr \cat{C})^{\hat{[1]}}_\Reedy$ by
  \Cref{semisimplicial-tribe-Reedy}.
  The subcategory $\frh \cat{C}$ contains the terminal object and
  is closed under pullbacks along fibrations.
  Moreover, it is closed under powers by finite semisimplicial sets by
  \Cref{frame-cotensor}.
  Therefore, it is enough to verify that it has factorizations.
  Given a morphism $a \to b$, we first factor $a_{1,0} \to b_{1,0}$
  as $a_{1,0} \weto \hat a_{1,0} \fto b_{1,0}$ in $\cat{C}$.
  If we set $\hat a_1$ to the canonical frame on $\hat a_{1,0}$, then
  $a_1 \to \hat a_1$ is a levelwise weak equivalence and
  $\hat a_1 \to b_1$ is a Reedy fibration by (SF).
  To complete the factorization, it suffices to
  factor $a_0 \to \hat a_1 \pull_{b_1} b_0$ into
  a levelwise weak equivalence and a Reedy fibration.

  The two evaluation functors $\frh \cat{C} \to \fr \cat{C}$ and
  $\frh \cat{C} \to \cat{C}$ are semisimplicial exact by construction.
  (In particular, preservation of powers by the latter follows from
  \Cref{frame-cotensor}.)

  In part (2) we proceed in similar manner,
  this time using (ST) to construct a factorization into
  a levelwise anodyne morphism followed by a Reedy fibration.
  Using this factorization and
  a retract argument as in the proof of \Cref{tribe-Reedy}
  we verify that a morphism of $\frh \cat{T}$ is anodyne \iff{}
  it is levelwise anodyne which directly implies the remaining axioms.
\end{proof}

\begin{lemma}\label{tribe-frh1}
  \leavevmode
  \begin{enumerate}
  \item If $\cat{C}$ is a semisimplicial fibration category,
  then the evaluation at $(1, 0)$ functor $\frh \cat{C} \to \cat{C}$ is
  a weak equivalence.
  \item If $\cat{T}$ is a semisimplicial tribe,
  then the evaluation at $(1, 0)$ functor $\frh \cat{T} \to \cat{T}$ is
  a weak equivalence.
  \end{enumerate}
\end{lemma}

\begin{proof}
  Part (2) is a special case of part (1) which is verified as follows.
  Define a functor $F \from \cat{C} \to \frh \cat{C}$ so that
  $(F a)_0 = (F a)_1 = a^{\ssimp{\uvar}}$ with the identity structure map.
  Then $F$ is a homotopical functor and
  $(F a)_{1,0} = a$ for all $a \in \cat{C}$.
  Moreover, for any $b \in \frh \cat{C}$
  the structure map $b_0 \to b_1$ provides a natural weak equivalence
  $b \weto F b_{1,0}$.
  Hence the evaluation functor is a homotopy equivalence.
\end{proof}

\begin{lemma}\label{tribe-frh2}
  \leavevmode
  \begin{enumerate}
  \item If $\cat{C}$ is a semisimplicial fibration category, then
   the evaluation at $0$ functor $\frh \cat{C} \to \fr \cat{C}$ is
   a weak equivalence.
  \item If $\cat{T}$ is a semisimplicial tribe, then
    the evaluation at $0$ functor $\frh \cat{T} \to \fr \cat{T}$ is
    a weak equivalence.  
  \end{enumerate}
\end{lemma}

\begin{proof}
  Part (2) follows from part (1) which is be proven as follows.
  The triangle
  \begin{tikzeq*}
  \matrix[diagram]
  {
                       & |(fhC)| \frh \cat{C} &               \\
    |(fC)| \fr \cat{C} &                      & |(C)| \cat{C} \\
  };

  \draw[->] (fhC) to node[above left]  {$\ev_0$}      (fC);
  \draw[->] (fhC) to node[above right] {$\ev_{1, 0}$} (C);
  \draw[->] (fC)  to node[below]       {$\ev_0$}      (C);
  \end{tikzeq*}
  commutes up to natural weak equivalence.
  Moreover, $\ev_{1, 0} \from \frh \cat{C} \to \cat{C}$ is a weak equivalence
  by \Cref{tribe-frh1} and
  $\ev_0 \from \fr \cat{C} \to \cat{C}$ is a weak equivalence
  by \Cref{tribe-frames}.
  Hence so is $\ev_0 \from \frh \cat{C} \to \fr \cat{C}$ by 2-out-of-3.
\end{proof}

\begin{proposition}\label{semisimp-DK}
  \leavevmode
  \begin{enumerate}
  \item The functor $\sFibCat \to \FibCat$ is a DK-equivalence.
  \item The functor $\sTrb \to \Trb$ is a DK-equivalence.
  \end{enumerate}
\end{proposition}

\begin{proof}
  The proofs of two parts are parallel.
  The functor $\fr \from \FibCat \to \sFibCat$ is homotopical.
  Indeed, this follows from the 2-out-of-3 property and the fact that
  the evaluation at $0$ functor $\fr \cat{C} \to \cat{C}$ is a weak equivalence
  as proven in \Cref{tribe-frames}.
  By \Cref{tribe-frames,tribe-frh1,tribe-frh2},
  it is a homotopy inverse to the functor $\sFibCat \to \FibCat$.
\end{proof}

%% file: 4-fibration-categories-of-fibration-categories-and-tribes.tex
\section{Fibration categories of fibration categories and tribes}
\label{fibcat-of-fibcats}

In this section, we construct the fibration categories of
semisimplicial fibration categories and semisimplicial tribes.
We begin by recalling the fibration category of fibration categories
of \cite{s1}.

\begin{definition}[\cite{s1}*{Definition 2.3}]\label{fibcat-fib}
  An exact functor $P \from \cat{E} \to \cat{D}$ between fibration categories
  is a \emph{fibration} if it satisfies the following properties.
  \begin{enumerate}
  \item It is an \emph{isofibration}: 
    for every object $a \in \cat{E}$ and an isomorphism $f' \from P a \to b'$
    there is an isomorphism $f \from a \to b'$ \st{} $P f = f'$.
  \item It has the \emph{lifting property for WF-factorizations}:
    for any morphism $f \from a \to b$ of $\cat{E}$ and a factorization
    \begin{tikzeq*}
    \matrix[diagram]
    {
      |(A)| P a &          & |(B)| P b \\
                & |(X)| c' &           \\
    };

    \draw[->]  (A) to node[above]       {$P f$} (B);
    \draw[fib] (X) to node[below right] {$p'$}   (B);

    \draw[->] (A) to node[below left] {$i'$} node[above right] {$\we$} (X);
    \end{tikzeq*}
    there exists a factorization
    \begin{tikzeq*}
    \matrix[diagram]
    {
      |(A)| a &         & |(B)| b \\
              & |(C)| c &         \\
    };

    \draw[->]  (A) to node[above]       {$f$} (B);
    \draw[fib] (C) to node[below right] {$p$} (B);

    \draw[->] (A) to node[below left] {$i$} node[above right] {$\we$} (C);
    \end{tikzeq*}
    \st{} $P i = i'$ and $P p = p'$ (in particular, $P c = c'$).
  \item It has the \emph{lifting property for pseudofactorizations}:
    for any morphism $f \from a \to b$ of $\cat{E}$ and a diagram
    \begin{tikzeq*}
    \matrix[diagram]
    {
      |(A)| P a & |(B)| P b \\
      |(X)| c'  & |(Y)| d'  \\
    };

    \draw[->]  (A) to node[above] {$P f$} (B);
    \draw[fib] (Y) to node[right] {$u'$}  (B);

    \draw[fib] (X) to node[left]  {$i'$} node[right] {$\we$} (A);
    \draw[->]  (X) to node[below] {$s'$} node[above] {$\we$} (Y);
    \end{tikzeq*}
    there exists a diagram
    \begin{tikzeq*}
    \matrix[diagram]
    {
      |(A)| a & |(B)| b \\
      |(C)| c & |(D)| d \\
    };

    \draw[->]  (A) to node[above] {$f$} (B);
    \draw[fib] (D) to node[right] {$u$} (B);

    \draw[fib] (C) to node[left]  {$i$} node[right] {$\we$} (A);
    \draw[->]  (C) to node[below] {$s$} node[above] {$\we$} (D);
    \end{tikzeq*}
    \st{} $P i = i'$, $P s = s'$ and $P u = u'$
    (in particular, $P c = c'$ and $P d = d'$).
  \end{enumerate}
\end{definition}

\begin{theorem}[\cite{s1}*{Thm.\ 2.8}]\label{fibcat-of-fibcats-old}
  The category of fibration categories with
  weak equivalences and fibrations as defined above is a fibration category.
  \qed
\end{theorem}

The key difficulty lies in the construction of path objects and
it is addressed in \cite{s1}*{Thm.\ 2.8} by
a modified version of the Reedy structure.
This modification is not available in the setting of tribes where
for a tribe $\cat{T}$ it is difficult to ensure that
both the path object $P \cat{T}$ is a tribe and
a homomorphism $\cat{T} \to P \cat{T}$ exists.

For semisimplicial tribes, we can use the standard Reedy structure to
construct $P \cat{T}$ and the homomorphism $\cat{T} \to P \cat{T}$ can
be defined using cotensors.

We proceed to define fibrations of semisimplicial tribes.
Since the definition does not depend on the enrichment,
we first give it for ordinary tribes.

\begin{definition}\label{tribe-fib}
  A homomorphism $P \from \cat{S} \to \cat{T}$ between tribes
  is a \emph{fibration} if it is a fibration of underlying fibration categories
  and satisfies the following properties.
  \begin{enumerate}
  \setcounter{enumi}{3}
  \item It has the \emph{lifting property for AF-factorizations}:
    for any morphism $f \from x \to y$ of $\cat{S}$ and a factorization
    \begin{tikzeq*}
    \matrix[diagram]
    {
      |(A)| P x &          & |(B)| P y \\
                & |(X)| z' &           \\
    };

    \draw[->]  (A) to node[above]       {$P f$} (B);
    \draw[fib] (X) to node[below right] {$p'$}  (B);

    \draw[cof] (A) to node[below left] {$i'$} node[above right] {$\we$} (X);
    \end{tikzeq*}
    there exists a factorization
    \begin{tikzeq*}
    \matrix[diagram]
    {
      |(A)| x &         & |(B)| y \\
              & |(C)| z &         \\
    };

    \draw[->]  (A) to node[above]       {$f$} (B);
    \draw[fib] (C) to node[below right] {$p$} (B);

    \draw[cof] (A) to node[below left] {$i$} node[above right] {$\we$} (C);
    \end{tikzeq*}
    \st{} $P i = i'$ and $P p = p'$ (in particular, $P z = z'$).
  \item It has the \emph{lifting property for lifts}: for a square
    \begin{tikzeq*}
    \matrix[diagram]
    {
      |(A)| a & |(B)| x \\
      |(C)| b & |(D)| y \\
    };

    \draw[->]  (A) to (B);
    \draw[fib] (B) to (D);
    \draw[->]  (C) to (D);

    \draw[cof] (A) to node[left] {$\we$} (C);
    \end{tikzeq*}
    in $\cat{S}$ and a lift for its image
    \begin{tikzeq*}
    \matrix[diagram,column sep={5em,between origins}]
    {
      |(A)| P a & |(B)| P x \\
      |(C)| P b & |(D)| P y \\
    };

    \draw[->]  (A) to (B);
    \draw[fib] (B) to (D);
    \draw[->]  (C) to (D);

    \draw[->,dashed] (C) to node[above left] {$f'$} (B);

    \draw[cof] (A) to node[left] {$\we$} (C);
    \end{tikzeq*}
    in $\cat{T}$,
    there exists a lift $f$ in the original square \st{} $P f = f'$.
  \item It has the \emph{lifting property for cofibrancy lifts}:
    for any acyclic fibration $p \from x \afto y$,
    a morphism $f \from a \to y$ in $\cat{S}$ and a lift
    \begin{tikzeq*}
    \matrix[diagram,column sep={5em,between origins}]
    {
                & |(x)| P x          \\
      |(a)| P a & |(y)| P y \text{,} \\
    };

    \draw[fib] (x) to node[left] {$\we$} node[right] {$P p$} (y);

    \draw[->] (a) to node[below] {$P f$} (y);

    \draw[->,dashed] (a) to node[above left] {$g'$} (x);
    \end{tikzeq*}
    there is a morphism $g \from a \to x$ \st{} $P g = g'$ and $p g = f$.
  \end{enumerate}
\end{definition}

This definition is similar to Joyal's definition of a \emph{meta-fibration},
i.e., a homomorphism of tribes satisfying conditions (1), (4) and (5) above
as well as the \emph{lifting property for sections of acyclic fibrations}:
for any acyclic fibration $p$ in $\cat{S}$ and a section $s$ of $P p$,
there is a section $s'$ of $p$ \st{} $P s' = s$.
The latter can be shown to be equivalent to (6).
Our definition also includes conditions (2) and (3) since
we need the forgetful functor $\sTrb \to \sFibCat$ to preserve fibrations.

\begin{definition}
  \leavevmode
  \begin{enumerate}
  \item An exact functor of semisimplicial fibration categories is
    a \emph{fibration} if it is a fibration of
    their underlying fibration categories.
  \item A homomorphism of semisimplicial tribes is a \emph{fibration} if it is
    a fibration of their underlying tribes.
  \end{enumerate}
\end{definition}

Recall that a functor $I \to J$ of small categories is a \emph{cosieve} if
it is injective on objects, fully faithful and
if $i \to j$ is a morphism of $J$ \st{} $i \in I$, then $j \in I$.
A cosieve between homotopical categories is additionally required to
preserve and reflect weak equivalences.
The following lemma gives a basic technique of constructing fibrations,
using cosieves.

\begin{lemma}\label{tribe-cosieve-fib}
  Let $I \to J$ be a cosieve between homotopical inverse categories.
  \begin{enumerate}
  \item If $\cat{C}$ is a semisimplicial fibration category, then
    the induced functor $\cat{C}^J_\Reedy \to \cat{C}^I_\Reedy$ is a fibration.
  \item If $\cat{T}$ is a semisimplicial tribe, then
    the induced functor $\cat{T}^J_\Reedy \to \cat{T}^I_\Reedy$ is a fibration.
  \end{enumerate}
\end{lemma}

\begin{proof}
  Part (1) follows directly from \cite{s1}*{Lemma 1.10}.

  In particular, the functor in part (2) is
  a fibration of underlying fibration categories.
  We check the remaining conditions of \Cref{tribe-fib}.

  Let $x \to y$ be a morphism in $\cat{T}^J_\Reedy$ and
  let $x|I \acto x' \fto y|I$ be a factorization of its restriction to $I$.
  By induction, it suffices to extend it to the subcategory generated by $I$
  and an object $j \in J$ of minimal degree among those not in $I$.
  The partial factorization above induces
  a morphism $x_j \to M_j x' \pull_{M_j y} y_j$ which
  we factor as an anodyne morphism $x_j \acto x'_j$ followed by
  a fibration $x'_j \fto M_j x' \pull_{M_j y} y_j$.
  This proves the lifting property for AF-factorizations.

  Let
  \begin{tikzeq*}
  \matrix[diagram]
  {
    |(a)| a & |(x)| x \\
    |(b)| b & |(y)| y \\
  };

  \draw[cof] (a) to node[left] {$\we$} (b);

  \draw[fib] (x) to (y);
  \draw[->]  (a) to (x);
  \draw[->]  (b) to (y);
  \end{tikzeq*}
  be a lifting problem in $\cat{T}^J_\Reedy$ and
  $b|I \to x|I$ a solution of its restriction to $I$.
  Again, it is enough to extend it to the subcategory generated by $I$
  and an object $j \in J$ of minimal degree among those not in $I$.
  This extension can be chosen as a solution in the following lifting problem:
  \begin{tikzeq*}
  \matrix[diagram,column sep={6em,between origins}]
  {
    |(a)| a_j & |(x)| x_j                              \\
    |(b)| b_j & |(y)| M_j x \pull_{M_j y} y_j \text{.} \\
  };

  \draw[cof] (a) to node[left] {$\we$} (b);

  \draw[fib] (x) to (y);
  \draw[->]  (a) to (x);
  \draw[->]  (b) to (y);
  \end{tikzeq*}
  This proves the lifting property for lifts.

  The verification of the lifting property for cofibrancy lifts is analogous.
\end{proof}

In the next two lemmas, we construct path objects and pullbacks along fibrations
in the categories $\sFibCat$ and $\sTrb$.

For a semisimplicial fibration category $\cat{C}$ let $P \cat{C}$ denote
$\cat{C}^{{\Sd \hat{[1]}}^\op}_\Reedy$, where $\Sd \hat{[1]}$ is
the homotopical poset $\{ 0 \weto 01 \weot 1 \}$.
It comes with a functor $\cat{C} \to P \cat{C}$ that maps $x$ to
the diagram $x \leftarrow x^{\ssimp{1}} \to x$
(which is semisimplicial exact by (SF)) and
a functor  $P \cat{C} \to \cat{C} \times \cat{C}$ that evaluates at $0$ and $1$.
For a semisimplicial tribe $\cat{T}$, we define
an analogous factorization $\cat{T} \to P \cat{T} \to \cat{T} \times \cat{T}$
in $\sTrb$
(where $\cat{T} \to P \cat{T}$ is a semisimplicial homomorphism
by (SF) and (ST)).

\begin{lemma}\label{tribe-path}
  \leavevmode
  \begin{enumerate}
  \item The object $P \cat{C}$ with
    the factorization $\cat{C} \to P \cat{C} \to \cat{C} \times \cat{C}$ is
    a path object for $\cat{C}$ in $\sFibCat$.
  \item The object $P \cat{T}$ with
    the factorization $\cat{T} \to P \cat{T} \to \cat{T} \times \cat{T}$ is
    a path object for $\cat{T}$ in $\sTrb$.
  \end{enumerate}
\end{lemma}

\begin{proof}
  The proofs of both parts are analogous so we only prove part (2).
  $P \cat{T}$ is a semisimplicial tribe by \Cref{semisimplicial-tribe-Reedy}.
  The evaluation $P \cat{T} \to \cat{T} \times \cat{T}$ is a fibration by
  \Cref{tribe-cosieve-fib}.
  The functor $\cat{T} \to P \cat{T}$ has a retraction given by
  evaluation at $0$ and thus it suffices to check that this evaluation is
  a weak equivalence.
  It is induced by a homotopy equivalence $[0] \to \Sd \hat{[1]}$ and
  hence it is a weak equivalence by \cite{s1}*{Lem.\ 1.8(3)}.
\end{proof}

\begin{lemma}\label{pullback-of-tribes}
  \leavevmode
  \begin{enumerate}
  \item Pullbacks along fibrations exist in $\sFibCat$.
  \item Pullbacks along fibrations exist in $\sTrb$.
  \end{enumerate}
\end{lemma}

\begin{proof}
  For part (1), let $F \from \cat{C} \to \cat{E}$ and
  $P \from \cat{D} \to \cat{E}$ be exact functors between
  semisimplicial fibration categories with $P$ a fibration.
  Form a pullback of $P$ along $F$ in the category of semisimplicial categories.
  \begin{tikzeq*}
    \matrix[diagram]
    {
      |(P)| \cat{P} & |(D)| \cat{D} \\
      |(C)| \cat{C} & |(E)| \cat{E} \\
    };

    \draw[->] (P) to node[above] {$G$} (D);
    \draw[->] (P) to node[left]  {$Q$} (C);

    \draw[->]  (C) to node[below] {$F$} (E);
    \draw[fib] (D) to node[right] {$P$} (E);
  \end{tikzeq*}
  Define a morphism $f$ of $\cat{P}$ to be a weak equivalence (a fibration) if
  both $G f$ and $Q f$ are weak equivalences (fibrations).
  By \cite{s1}*{Prop.\ 2.4}, $\cat{P}$ is
  a pullback of the underlying fibration categories.
  It remains to verify that it is a semisimplicial fibration category.
  
  Let $(x, y)$ be an object of $\cat{P}$ and $K$ a finite semisimplicial set.
  We form a cotensor $x^K$ in $\cat{C}$ and lift its image $F(x^K)$ in $\cat{E}$
  to a cotensor $y^K$ in $\cat{D}$ using the fact that $P$ is an isofibration.
  Then $(x^K, y^K)$ is a cotensor of $(x, y)$ by $K$ in $\cat{P}$.
  The pullback-cotensor property (SF) is satisfied in $\cat{P}$ since
  it is satisfied in both $\cat{C}$ and $\cat{D}$.

  The proof of part (2) is very similar.
  However, there are a few differences, so we spell it out.

  Let $F \from \cat{R} \to \cat{T}$ and $P \from \cat{S} \to \cat{T}$
  be homomorphisms between semisimplicial tribes with $P$ a fibration.
  Form a pullback of $P$ along $F$ in the category of semisimplicial categories.
  \begin{tikzeq*}
    \matrix[diagram]
    {
      |(P)| \cat{P} & |(S)| \cat{S} \\
      |(R)| \cat{R} & |(T)| \cat{T} \\
    };

    \draw[->] (P) to node[above] {$G$} (S);
    \draw[->] (P) to node[left]  {$Q$} (R);

    \draw[->]  (R) to node[below] {$F$} (T);
    \draw[fib] (S) to node[right] {$P$} (T);
  \end{tikzeq*}
  Define a morphism $f$ of $\cat{P}$ to be a fibration if
  both $G f$ and $Q f$ are fibrations.

  First, let $1_\cat{R}$ be a terminal object of $\cat{R}$.
  Since $P$ is an isofibration
  there is a terminal object $1_\cat{S}$ of $\cat{S}$
  \st{} $P 1_\cat{S} = F 1_\cat{R}$.
  Then $(1_\cat{R}, 1_\cat{S})$ is a terminal object of $\cat{P}$.
  Moreover, all objects are fibrant since all objects of $\cat{R}$ and $\cat{S}$
  are fibrant.
  This proves \Cref{tribe-terminal}.

  Similarly, to construct a pullback in $\cat{P}$,
  we first construct it in $\cat{R}$ and
  then lift its image from $\cat{T}$ to $\cat{S}$.
  Fibrations are stable under pullback since
  they are stable in both $\cat{R}$ and $\cat{S}$.
  This proves \Cref{tribe-pullback}.

  To factor a morphism as a levelwise anodyne morphism
  (i.e., one whose components in $\cat{R}$ and $\cat{S}$ are anodyne)
  followed by a fibration in $\cat{P}$, we first factor it in $\cat{R}$ and
  then lift the image of this factorization from $\cat{T}$ to $\cat{S}$.
  A retract argument as in the proof of \Cref{tribe-Reedy} shows that
  every anodyne morphism is levelwise anodyne.

  Conversely, to solve a lifting problem between
  a levelwise anodyne morphism and a fibration, we first solve it in $\cat{R}$
  and then lift the image of this solution from $\cat{T}$ to $\cat{S}$.
  It follows that every morphism factors as
  an anodyne morphism followed by a fibration,
  proving \Cref{tribe-factorization}.

  Levelwise anodyne morphisms are stable under pullbacks along fibrations
  and thus so are the anodyne morphisms.
  This proves \Cref{tribe-anodyne}.

  Before we can conclude the proof, we need to verify that
  a morphism $f$ in $\cat{P}$ is a weak equivalence if and only if
  $Q f$ an $G f$ are weak equivalences in $\cat{R}$ and $\cat{S}$, respectively.
  To this end, we consider the fibration category $\cat{P}_{\mathrm{lvl}}$
  arising as the pullback of the underlying fibration categories of
  $\cat{R}$, $\cat{S}$ and $\cat{T}$.
  Note that $\cat{P}$ and $\cat{P}_{\mathrm{lvl}}$ have
  the same underlying category and the same fibrations while
  weak equivalences in $\cat{P}_{\mathrm{lvl}}$ are levelwise.
  Moreover, every object of $\cat{P}_{\mathrm{lvl}}$ is cofibrant.
  Indeed, to find a lift against an acyclic fibration $p$
  in $\cat{P}_{\mathrm{lvl}}$, using \Cref{tribe-cofibrant},
  we first pick a lift against $Q p$ in $\cat{R}$
  and then lift its image in $\cat{T}$ to a lift against $G p$ in $\cat{S}$.
  A path object in $\cat{P}$ is also a path object in $\cat{P}_{\mathrm{lvl}}$
  since anodyne morphisms in $\cat{P}$ are levelwise.
  The conclusion follows by
  the same argument as in the proof of \Cref{tribe-Reedy}.

  Finally, since weak equivalences in $\cat{P}$ are levelwise,
  the proof that it is a semisimplicial tribe is the same as
  in the case of fibration categories above.
\end{proof}

In the construction of the fibration categories $\sFibCat$ and $\sTrb$,
we will need the following characterization of acyclic fibrations.

\begin{lemma}\label{tribe-acyclic-fib}
  \leavevmode
  \begin{enumerate}
  \item An exact functor $P \from \cat{C} \to \cat{D}$ of
    semisimplicial fibrations categories is an acyclic fibration \iff{}
    it is a fibration, satisfies (App1) and
    for every fibration $q \from x' \fto P y$ in $\cat{D}$,
    there is a fibration $p \from x \fto y$ \st{} $P p = q$.
  \item A homomorphism $P \from \cat{S} \to \cat{T}$ of semisimplicial tribes is
    an acyclic fibration \iff{} it is a fibration, satisfies (App1) and
    for every fibration $q \from x' \fto P y$ in $\cat{T}$,
    there is a fibration $p \from x \fto y$ \st{} $P p = q$.
  \end{enumerate}
\end{lemma}

\begin{proof}
  This follows directly from \cite{s1}*{Prop.\ 2.5}.
\end{proof}

\begin{theorem}\label{fibcat-of-tribes}
  \leavevmode
  \begin{enumerate}
  \item The category of semisimplicial fibration categories with
    weak equivalences and fibrations as defined above is a fibration category.
  \item The category of semisimplicial tribes with
    weak equivalences and fibrations as defined above is a fibration category.
  \end{enumerate}
\end{theorem}

\begin{proof}
  The proofs of both parts are parallel.
  Axioms \Cref{fibcat-terminal,fibcat-2-out-of-6} are immediate.
  \Cref{fibcat-pullback} follows from \Cref{pullback-of-tribes,tribe-acyclic-fib}
  while \Cref{fibcat-factorization} follows from \Cref{tribe-path} and
  \cite{br}*{Factorization lemma, p.\ 421}.
\end{proof}

%% file: 5-presheaves-over-simplicial-categories.tex
\section{Presheaves over simplicial categories}
\label{presheaves}

In this section we introduce
a tribe of injectively fibrant simplicial presheaves over a simplicial category,
which will be the starting point of constructions of tribes in
\Cref{hammocks,tribes-of-presheaves}.
We begin by recalling
model structures on the categories of simplicial presheaves.

\begin{theorem}[\cite{HTT}*{Prop.\ A.3.3.2}]
  The category of simplicial (enriched) presheaves over
  a small simplicial category carries
  two cofibrantly generated proper model structures:
  \begin{enumerate}
  \item the \emph{injective} model structure where
    weak equivalences are levelwise weak homotopy equivalences and
    cofibrations are monomorphisms;
  \item the \emph{projective} model structure where
    weak equivalences are levelwise \whe{}s and
    fibrations are levelwise Kan fibrations. \qed
  \end{enumerate}
\end{theorem}

The fibrations of these model structures are usually called
\emph{injective fibrations} and \emph{projective fibrations}, respectively.
We will almost always use injective fibrations,
so we will call them \emph{fibrations} for brevity.
Similarly, \emph{injectively fibrant} presheaves will be referred to as
\emph{fibrant} presheaves.

The cofibrations of the injective model structure are closed under pullback
since they are exactly monomorphisms.
This yields the following corollary.

\begin{corollary}\label{presheaves-tribe}
  The category of fibrant presheaves over a small simplicial category
  is a tribe. \qed
\end{corollary}

In the remainder of this section,
we review a few standard facts about homotopy theory of presheaves.

We will use the notion of a homotopy pullback as in \Cref{hpb}
in the fibration category underlying the tribe above.
However, since it arises from a right proper model structure,
\Cref{hpb} applies verbatim even to non-fibrant presheaves.
In particular, \Cref{hpb-invariant} holds for non-fibrant presheaves as well.

\begin{lemma}\label{hpb-fibers}
  A square
  \begin{tikzeq*}
  \matrix[diagram]
  {
    |(U)| A & |(X)| C \\
    |(V)| B & |(Y)| D \\
  };

  \draw[->] (U) to (V);
  \draw[->] (X) to (Y);

  \draw[->] (U) to (X);
  \draw[->] (V) to (Y);
  \end{tikzeq*}
  of presheaves over $\cat{A}$ is a homotopy pullback \iff{}
  for every object $a \in \cat{A}$ and every point of $B_a$,
  the induced map from the homotopy fiber of $A \to B$ to
  the homotopy fiber of $C \to D$ is a \whe{}.
\end{lemma}

\begin{proof}
  Since every injective fibration is in particular a projective fibration,
  the square above is a homotopy pullback \iff{} it is
  a levelwise homotopy pullback.
  Moreover, by \cite{mv}*{Prop.\ 3.3.18} a square of simplicial sets is
  a homotopy pullback \iff{} the induced maps on all homotopy fibers are
  \whe{}s.
\end{proof}

\begin{definition}
  If $\cat{A}$ is a simplicial category, then a presheaf $A$ over $\cat{A}$ is
  \emph{homotopy representable} if there exist $a \in \cat{A}$ and
  a weak equivalence $\cat{A}(\uvar, a) \weto A$
  (called a \emph{representation} of $A$).
\end{definition}

% For brevity, homotopy representable presheaves will be called representable.

\begin{lemma}\label{representable-invariant}
  A presheaf weakly equivalent to a homotopy representable one is
  also homotopy representable.
\end{lemma}

\begin{proof}
  If $A \weto B$ is a weak equivalence and $A$ is homotopy representable,
  then so is $B$.
  Thus it is enough to check that if $B$ is homotopy representable,
  then so is $A$.
  Let $r_b \from \cat{A}(\uvar, b) \weto B$ be a representation.
  Since $A_b \to B_b$ is a \whe{}, we can pick a vertex in $ A_b$ and
  a path connecting its image in $B_b$ to $r_b(\id_b)$ which yields
  a homotopy commutative triangle
  \begin{tikzeq*}
  \matrix[diagram,column sep={7em,between origins}]
  {
                            & |(A)| A \\
    |(b)| \cat{A}(\uvar, b) & |(B)| B \\
  };

  \draw[->,dashed] (b) to (A);

  \draw[->] (A) to node[left] {$\we$} (B);

  \draw[->] (b) to node[above] {$r_b$} node[below] {$\we$} (B);
  \end{tikzeq*}
  and thus the dashed arrow provides a representation of $A$.
\end{proof}

\begin{lemma}[\cite{hi}*{Cor.\ 7.3.12(2)}]\label{HEP}
  If $X \fto Y$ is a fibration between presheaves and
  \begin{tikzeq*}
  \matrix[diagram,column sep={6em,between origins}]
  {
    |(A)| A &         \\
    |(X)| X & |(Y)| Y \\
  };

  \draw[->] (A) to node[left] {$f$} (X);

  \draw[->]  (A) to (Y);
  \draw[fib] (X) to (Y);
  \end{tikzeq*}
  is a homotopy commutative triangle, then
  there is a map $g \from A \to X$ homotopic to $f$
  making the triangle commute strictly. \qed
\end{lemma}

%% file: 6-hammock-localization-of-a-fibration-category.tex
\section{Hammock localization of a fibration category}
\label{hammocks}

The goal of the present section is the construction of
the tribe of homotopy representable fibrant presheaves over
the hammock localization $\hamm \cat{C}$ \cite{dk2} of
a fibration category $\cat{C}$.
In the context of fibration categories,
the mapping spaces of the hammock localization can be approximated by
the categories of fractions.

In a fibration category $\cat{C}$, a \emph{right fraction} from $a$ to $b$ is
a diagram of the form
\begin{tikzeq*}
\matrix[diagram]
{
  |(a)| a & |(e)| a' & |(b)| b \\
};

\draw[->] (e) to node[above] {$s$} (b);

\draw[->] (e) to node[above] {$v$} node[below] {$\we$} (a);
\end{tikzeq*}
which we denote by $s \bar{v} \from a \rfto b$.
Such a fraction is \emph{Reedy fibrant} if the morphism $a' \to a \times b$
is a fibration.
We will write $\Frac_{\cat{C}}(a, b)$ for
the category of right fractions from $a$ to $b$
and $\Frac^\Reedy_{\cat{C}}(a, b)$ for
the category of Reedy fibrant right fractions from $a$ to $b$.

\begin{theorem}\label{fractions-hammocks}
  For all $a, b \in \cat{C}$ the canonical maps
  \begin{tikzeq*}
  \matrix[diagram,column sep={8em,between origins}]
  {
    |(FR)| \nerve \Frac^\Reedy_{\cat{C}}(a, b) &[1em]
    |(F)|  \nerve \Frac_{\cat{C}}(a, b)        &
    |(H)|  \hamm \cat{C}(a, b)                 \\
  };

  \draw[->] (FR) to node[above] {$\iota$} (F);

  \draw[->] (F) to (H);
  \end{tikzeq*}
  are \whe{}s.
\end{theorem}

This theorem is a variant of a classical result of Dwyer and Kan
\cite{dk2}*{Prop.\ 6.2}.
The proof of the first part uses semisimplicial methods which
we summarize in the following lemmas.

\begin{lemma}\label{simp-semisimp}
  The unit and the counit of the free/forgetful adjunction $F \adj U$ are
  both \whe{}s.
  Hence the forgetful functor from simplicial sets to semisimplicial sets is
  a homotopy equivalence.
\end{lemma}

\begin{proof}
  First, note that $F U \simp{m}$ is the nerve of the category $[m]'$
  whose morphisms are exactly the morphisms of $[m]$ along with
  one idempotent endomorphism of each object.
  There is a natural transformation connecting $\id \from [m]' \to [m]'$ with
  the constant functor at $0$ and hence $F U \simp{m}$ is contractible.
  Thus $\epsilon_{\simp{m}} \from F U \simp{m} \to \simp{m}$ is
  a \whe{} for all $m$.
  Since $F$ and $U$ preserve colimits (they are both left adjoints),
  a standard inductive argument over the skeleta of a simplicial set $K$
  (using the Gluing Lemma in the inductive step) shows that
  $\epsilon$ is a \whe{} everywhere.
  The \whe{}s of semisimplicial sets are created by $F$ and thus
  the triangle identity $\epsilon F \cdot F \eta = \id_F$ implies that
  $\eta$ is also a \whe{}.
\end{proof}

A \emph{homotopy} between semisimplicial maps $f, g \from K \to L$ is
a semisimplicial map $K \gprod \ssimp{1} \to L$ that
extends $[f, g] \from K \gprod \bdssimp{1} \to L$.
The free simplicial set functor is monoidal and hence
it carries such semisimplicial homotopies to simplicial homotopies.

\begin{lemma}\label{HELP}
  Let $K \to L$ be a semisimplicial map \st{} for every $m$,
  every square of the form
  \begin{tikzeq*}
  \matrix[diagram,column sep={6em,between origins}]
  {
    |(b)| \bdssimp{m} & |(K)| K \\
    |(s)| \ssimp{m}   & |(L)| L \\
  };

  \draw[inj] (b) to (s);
  \draw[->]  (K) to (L);
  \draw[->]  (b) to (K);
  \draw[->]  (s) to (L);
  \end{tikzeq*}
  admits a lift up to homotopy, i.e.,
  there is a map $\ssimp{m} \to K$ making the upper triangle commute strictly
  and the lower triangle commute up to homotopy relative to $\bdssimp{m}$.
  Then $K \to L$ is a \whe{}.
\end{lemma}

\begin{proof}
  The argument of \cite{ks}*{Lem.\ 5.3} shows that $K \to L$ is
  a semisimplicial homotopy equivalence.
  The preceding remark implies that
  the free functor carries it to a simplicial homotopy equivalence and
  thus it is a \whe{}.
\end{proof}

\begin{proof}[Proof of \Cref{fractions-hammocks}]
%  The second morphism is a \whe{} by \cite{nss}*{Thm.\ 3.61}.
%
  To verify that $\iota$ is a \whe{}, by Quillen's Theorem A \cite{q}*{p.\ 85}
  and \cite{ltw}*{Thm.\ 4.1 and Rmk.\ 5.6} it is enough to show that
  for every fraction $s \bar{v} \in \Frac_{\cat{C}}(a, b)$
  the slice $\Ex \nerve (s \bar{v} \slice \iota)$ is contractible.
  Here, $\Ex$ denotes Kan's classical functor \cite{kan}.

  We define a semisimplicial set $\Ex \nerve (s \bar{v} \slice \iota)^\Reedy$
  as follows.
  An $m$-simplex in $\Ex \nerve (s \bar{v} \slice \iota)^\Reedy$ is
  an $m$-simplex $(\Sd[m])^\op \to s \bar{v} \slice \iota$ of
  $\Ex \nerve (s \bar{v} \slice \iota)$ that is Reedy fibrant.
  Here for a poset $P$, $\Sd P$ denotes the poset of chains
  (i.e., finite non-empty totally ordered subsets) in $P$, ordered by inclusion.
  Explicitly, such an $m$-simplex consists of
  a functor $t \bar{w} \from (\Sd[m])^\op \to \Frac_{\cat{C}}(a, b)$
  Reedy fibrant as a diagram in $\cat{C} \fslice a \times b$ together with
  a morphism $\lambda \from s \bar{v} \to t_{[m]} \bar{w}_{[m]}$.

  We will check that the inclusion
  $\Ex \nerve (s \bar{v} \slice \iota)^\Reedy \ito U \Ex \nerve (s \bar{v} \slice \iota)$
  is a \whe{}.
  By \Cref{HELP}, it is enough to find a lift up to homotopy
  in every diagram of the form
  \begin{tikzeq*}
  \matrix[diagram,column sep={8em,between origins}]
  {
    |(b)| \bdssimp{m} & |(ExR)|   \Ex \nerve (s \bar{v} \slice \iota)^\Reedy   \\
    |(s)| \ssimp{m}   & |(Ex)|  U \Ex \nerve (s \bar{v} \slice \iota) \text{.} \\
  };

  \draw[inj] (b)   to (s);
  \draw[->]  (ExR) to (Ex);
  \draw[->]  (b)   to (ExR);
  \draw[->]  (s)   to (Ex);
  \end{tikzeq*}
  The data of the diagram corresponds to
  a simplex $(t \bar{w}, \lambda)$ as above \st{}
  $t \bar{w}$ is Reedy fibrant only over $\Sd \bdssimp{m}$, i.e.,
  the subposet of $\Sd [m]$ obtained by removing the top element.
  By \cite{s2}*{Lem.\ 1.9(1)} we can find a Reedy fibrant $t' \bar{w}'$
  together with a weak equivalence $u \from t \bar{w} \weto t' \bar{w}'$ that
  restricts to the identity over $\Sd \bdssimp{m}$.
  Then $(t' \bar{w}', u_{[m]} \lambda)$ is
  a simplex $\ssimp{m} \to \Ex \nerve (s \bar{v} \slice \iota)^\Reedy$ which
  makes the upper triangle commute while $u$ yields a homotopy in the lower one
  by taking the composite
  \begin{tikzeq*}
  \matrix[diagram,column sep={8em,between origins}]
  {
    |(Sd)|  (\Sd([m] \times [1]))^\op &[3em]
    |(Sd1)| (\Sd[m])^\op \times [1]   &
    |(svi)| s \bar{v} \slice \iota    \\
  };

  \draw[->] (Sd)  to (Sd1);

  \draw[->] (Sd1) to node[above] {$u$} (svi);
  \end{tikzeq*}
  where the first map takes a chain $A \subseteq [m] \times [1]$ to
  $(\mathrm{proj}_0 A, \min \mathrm{proj}_1 A)$.

  \Cref{simp-semisimp} implies that
  $\Ex \nerve (s \bar{v} \slice \iota)$ is contractible \iff{}
  $\Ex \nerve (s \bar{v} \slice \iota)^\Reedy$ is contractible.
  Thus it suffices to show that
  $\Ex \nerve (s \bar{v} \slice \iota)^\Reedy$ is a contractible Kan complex.
  To extend a map
  $(t \bar{w}, \lambda) \from \bdssimp{m} \to \Ex \nerve (s \bar{v} \slice \iota)^\Reedy$
  to $\ssimp{m}$ we factor $s \bar{v} \to \lim_{(\Sd \bdsimp{m})^\op} t \bar{w}$
  as a weak equivalence $s \bar{v} \to t_{[m]} \bar{w}_{[m]}$ followed by
  a fibration.

  The morphism $\nerve \Frac_{\cat{C}}(a, b) \to \hamm \cat{C}(a, b)$ is
  a \whe{} by \cite{nss}*{Thm.\ 3.61} provided that
  $\cat{C}$ has functorial factorizations.
  We will generalize this to an arbitrary fibration category.
  It follows from \cite{s3}*{Lem.\ 4.5} that every fibration category is
  weakly equivalent to one with functorial factorizations
  (cf.\ \cite{s3}*{Def.\ 4.1}).
  Since \Cref{App} implies that
  a weak equivalence $F \from \cat{C} \to \cat{D}$ induces
  a \whe{} $\hamm \cat{C}(a, b) \to \hamm \cat{D}(F a, F b)$,
  it is be enough to verify that it also induces
  a \whe{} $\nerve \Frac_{\cat{C}}(a, b) \to \nerve \Frac_{\cat{D}}(F a, F b)$.
  By the first part of the proof, it suffices to check that
  $\nerve \Frac^\Reedy_{\cat{C}}(a, b) \to \nerve \Frac^\Reedy_{\cat{D}}(F a, F b)$
  is a \whe{}.

  Just as above, we define $\nerve \Frac^\Reedy_{\cat{C}}(a, b)^\Reedy$ to be
  the semisimplicial subset of $U \nerve \Frac^\Reedy_{\cat{C}}(a, b)$
  consisting of these simplices
  that are Reedy fibrant as diagrams in $\cat{C} \fslice a \times b$.
  An analogous argument shows that the inclusion
  $\nerve \Frac^\Reedy_{\cat{C}}(a, b)^\Reedy \ito U \nerve \Frac^\Reedy_{\cat{C}}(a, b)$
  is a \whe{}.
  Thus it remains to show that
  $\nerve \Frac^\Reedy_{\cat{C}}(a, b)^\Reedy \to \nerve \Frac^\Reedy_{\cat{D}}(F a, F b)^\Reedy$
  is a \whe{} for which we will verify the assumptions of \Cref{HELP}.

  Consider a diagram
  \begin{tikzeq*}
  \matrix[diagram,column sep={9em,between origins}]
  {
    |(b)| \bdssimp{m} & |(C)| \nerve \Frac^\Reedy_{\cat{C}}(a, b)^\Reedy     \\
    |(s)| \ssimp{m}   & |(D)|  \nerve \Frac^\Reedy_{\cat{D}}(F a, F b)^\Reedy \\
  };

  \draw[inj] (b) to (s);

  \draw[->] (C) to node[right] {$F$} (D);

  \draw[->] (b) to node[above] {$c$} (C);
  \draw[->] (s) to node[below] {$z$} (D);
  \end{tikzeq*}
  where $c$ and $z$ denote the middle objects in the diagrams of fractions
  (the morphism of fractions themselves are suppressed from the notation).
  We apply (App2) to the morphism $z_{[m]} \to F M_{[m]} c$ and
  obtain a square
  \begin{tikzeq*}
  \matrix[diagram]
  {
    |(z)|  z  & |(c)|  F M_{[m]} c \\
    |(z')| z' & |(c')| F c_{[m]} \\
  };

  \draw[->] (z') to node[left]  {$\we$} (z);
  \draw[->] (z') to node[below] {$\we$} (c');

  \draw[->] (z) to (c);
  \draw[->] (c') to (c);
  \end{tikzeq*}
  where, using the factorization axiom and (App1), we can assume that
  the right morphism is induced by a fibration $c_{[m]} \fto F M_{[m]} c$.
  This extends $c$ to a Reedy fibrant diagram over $\Sd [m]$.
  The diagrams $z$ and $F c$ together with $z'$ assemble to
  a diagram over $\Sd [m] \times \Sd [1]$ which we pull back to
  $\Sd ([m] \times [1])$ along the functor induced by
  the projections $[m] \times [1] \to [m]$ and $[m] \times [1] \to [1]$.
  Applying \cite{s2}*{Lem.\ 1.9(1)} in $\cat{C} \fslice a \times b$
  allows us to replace it by a Reedy fibrant diagram without
  changing it over $\Sd ([m] \times \{ 0, 1 \})$ or
  $\Sd(\bdssimp{m} \gprod \ssimp{1})$.
  This replacement yields a homotopy as required in \Cref{HELP} which
  concludes the proof.
\end{proof}

As a consequence we obtain the following corollary.

\begin{corollary}\label{tribe-single-arrow}
  If $\cat{T}$ is a tribe, then every zig-zag in $\hamm \cat{T}$ is
  homotopic to a single morphism in $\cat{T}$.
\end{corollary}

\begin{proof}
  By \Cref{fractions-hammocks}, it suffices to show that
  every fraction is homotopic to a single arrow.
  Let
  \begin{tikzeq*}
  \matrix[diagram]
  {
    |(x)| x & |(x')| x' & |(y)| y \\
  };

  \draw[->] (x') to node[above] {$v$} node[below] {$\we$} (x);

  \draw[->] (x') to node[above] {$s$} (y);
  \end{tikzeq*}
  be such a fraction.
  Since $v$ is a weak equivalence, we can choose
  a homotopy inverse $w \from x \to x'$.
  Let $H \from x' \to P x'$ be a homotopy from $w v$ to $\id_{x'}$, where
  $(\pi_0, \pi_1) \sigma \from x' \to P x' \to x' \times x'$ is a path object.
  Then
  \begin{tikzeq*}
  \matrix[diagram,column sep={5em,between origins},row sep={5em,between origins}]
  {
             & |(x1)|   x  & |(x'20)| x'   & |(x'30)| x' &         \\
    |(x0)| x & |(x'11)| x' & |(Px')|  P x' & |(x'31)| x' & |(y)| y \\
             & |(x'12)| x' & |(x'22)| x'   & |(x'32)| x' &         \\
  };

  \draw[->] (x0)   to (x1);
  \draw[->] (x'20) to node[above] {$v$} (x1);
  \draw[->] (x'20) to (x'30);
  \draw[->] (x'30) to node[above right] {$s$} (y);

  \draw[->] (x0)   to node[above] {$w$} (x'11);
  \draw[->] (Px')  to node[above] {$\pi_0$} (x'11);
  \draw[->] (Px')  to node[above] {$\pi_1$} (x'31);
  \draw[->] (x'31) to node[above] {$s$} (y);

  \draw[->] (x0)   to node[below left] {$w$} (x'12);
  \draw[->] (x'22) to (x'12);
  \draw[->] (x'22) to (x'32);
  \draw[->] (x'32) to node[below right] {$s$} (y);

  \draw[->] (x1)   to node[right] {$w$} (x'11);
  \draw[->] (x'12) to (x'11);

  \draw[->] (x'20) to node[right] {$H$} (Px');
  \draw[->] (x'22) to node[right] {$\sigma$} (Px');

  \draw[->] (x'30) to (x'31);
  \draw[->] (x'32) to (x'31);
  \end{tikzeq*}
  (where the unlabeled arrows are identities)
  is a homotopy between the original fraction and
  the composite $s w$ in $\hamm \cat{T}$.
\end{proof}

We next turn our attention to homotopy limits in
the category of homotopy representable fibrant presheaves over $\hamm \cat{C}$.
In \Cref{hammock-terminal} we characterize homotopy terminal objects and
in \Cref{hammock-pullback} homotopy pullbacks.

\begin{lemma}\label{Yoneda-we}
  A morphism $f \from a \to b$ in $\cat{C}$ is a weak equivalence \iff{}
  the induced map
  \begin{tikzeq*}
  \matrix[diagram,column sep={8em,between origins}]
  {
    |(a)| \hammrep{\cat{C}, a} & |(b)| \hammrep{\cat{C}, b} \\
  };
  
  \draw[->] (a) to node[above] {$f_*$} (b);
  \end{tikzeq*}
  is a weak equivalence.
\end{lemma}

\begin{proof}
  If $f$ is a weak equivalence, then so is $f_*$ by \cite{dk2}*{Prop.\ 3.3}.

  Conversely, the map $\pi_0 \hammrep{\cat{C}, a} \to \pi_0 \hammrep{\cat{C}, b}$
  coincides with $\Ho \cat{C}(\uvar, a) \to \Ho \cat{C}(\uvar, b)$ which
  is therefore an isomorphism.
  Thus $f$ is an isomorphism in $\Ho \cat{C}$ and
  it follows from \cite{rb}*{Thm.\ 7.2.7} that it is a weak equivalence.
\end{proof}

\begin{lemma}\label{hammock-terminal}
  For every $a \in \cat{C}$ the morphism $a \to 1$ is a weak equivalence \iff{}
  $\hamm \cat{C} (\uvar, a) \to 1$ is a weak equivalence of presheaves.
\end{lemma}

\begin{proof}
  First assume that $a \to 1$ is a weak equivalence.
  For every $e \in \cat{C}$ we have weak homotopy equivalences
  \begin{tikzeq*}
  \matrix[diagram,column sep={8em,between origins}]
  {
    |(a)| \hamm \cat{C}(e, a) &
    |(H)| \hamm \cat{C}(e, 1) &
    |(F)| \nerve \Frac_{\cat{C}}(e, 1) \\
  };

  \draw[->] (a) to node[above] {$\we$} (H);
  \draw[->] (F) to node[above] {$\we$} (H);
  \end{tikzeq*}
  by \Cref{fractions-hammocks} and \Cref{Yoneda-we}.
  Thus it is enough to check that $\Frac_{\cat{C}}(e, 1)$ is contractible which
  is the case since it has a terminal object $e \leftarrow e \to 1$.

  Conversely, assume that
  $\hamm \cat{C}(\uvar, a) \to 1$ is a weak equivalence.
  Then in the diagram
  \begin{tikzeq*}
  \matrix[diagram]
  {
    |(La)| \hamm \cat{C}(\uvar, a) & & |(L1)| \hamm \cat{C}(\uvar, 1) \\
    & |(1)| 1 \\
  };

  \draw[->] (La) to (L1);

  \draw[->] (La) to node[below left]  {$\we$} (1);
  \draw[->] (L1) to node[below right] {$\we$} (1);
  \end{tikzeq*}
  both downward maps are weak equivalences by the first part of the proof and
  hence so is the horizontal one.
  It follows that $a \to 1$ is a weak equivalence by \Cref{Yoneda-we}.
\end{proof}

\begin{lemma}\label{strict-representation}
  If $q \from A \fto B$ is a fibration between
  homotopy representable presheaves over $\hamm \cat{C}$, then
  for every fixed representation $r_b \from \hamm \cat{C}(\uvar, b) \weto B$
  there are a representation $r_a \from \hamm \cat{C}(\uvar, a) \weto A$ and
  a fibration $p \from a \fto b$ \st{} the square
  \begin{tikzeq*}
  \matrix[diagram,column sep={6em,between origins}]
  {
    |(a)| \hamm \cat{C}(\uvar, a) & |(A)| A \\
    |(b)| \hamm \cat{C}(\uvar, b) & |(B)| B \\    
  };

  \draw[->] (a) to node[above] {$r_a$} (A);
  \draw[->] (b) to node[below] {$r_b$} (B);

  \draw[->] (a) to node[left]  {$p_*$} (b);
  \draw[->] (A) to node[right] {$q$}   (B);
  \end{tikzeq*}
  commutes.
\end{lemma}

\begin{proof}
  Pick some
  representation $r_{\tilde a} \from \hammrep{\cat{C},\tilde a} \weto A$ and
  consider $q r_{\tilde a} (\id_{\tilde a}) \in B_{\tilde a}$.
  Since both
  \begin{tikzeq*}
  \matrix[diagram,column sep={7em,between origins}]
  { 
    |(F)| \nerve \Frac^\Reedy_{\cat{C}}(\tilde a, b) &[2em]
    |(H)| \hamm \cat{C}(\tilde a, b)                 &
    |(B)| B_{\tilde a}                               \\
  };
  
  \draw[->] (F) to (H);

  \draw[->] (H) to node[above] {$r_b$} (B);
  \end{tikzeq*}
  are weak homotopy equivalences (see \Cref{fractions-hammocks}) and
  $B_{\tilde a}$ is a Kan complex, there are
  a fraction $p \bar{w} \from \tilde a \rfto b$ and
  an edge in $B_{\tilde a}$ connecting
  $r_b (p \bar{w})$ to $q r_{\tilde a} (\id_{\tilde a})$.
  This yields a diagram
  \begin{tikzeq*}
  \matrix[diagram,column sep={8em,between origins}]
  {
    |(a)| \hammrep{\cat{C},a} & |(A)| A \\
    |(b)| \hammrep{\cat{C},b} & |(B)| B \\
  };

  \draw[->] (a) to node[left]  {$p_*$} (b);
  \draw[->] (A) to node[right] {$q$}   (B);

  \draw[->] (a)  to node[above] {$r_{\tilde a} w_*$} (A);
  \draw[->] (b)  to node[below] {$r_b$}              (B);
  \end{tikzeq*}
  (where $a$ is the domain of $p$ and $w$)
  which does not commute strictly but only up to homotopy.
  However, since $q$ is a fibration this square can be strictified
  for the price of replacing $r_{\tilde a} w_*$ by homotopic $r_a$
  by \Cref{HEP}.
  Since both $w_*$ and $r_{\tilde a}$ are weak equivalences
  (by \Cref{Yoneda-we}) so is $r_a$ which completes the proof.
\end{proof}

\begin{lemma}\label{fiber-representables}
  Let $p \from a \fto b$ be a fibration in $\cat{C}$, then
  for every $e \in \cat{C}$ and
  a Reedy fibrant fraction $s \bar{v} \from e \rfto b$
  the slice $p_* \slice s \bar{v}$ is the homotopy fiber of
  $p_* \from \Frac^\Reedy_{\cat{C}}(e, a) \to \Frac^\Reedy_{\cat{C}}(e, b)$
  over $s \bar{v}$.
\end{lemma}

\begin{proof}
  By Quillen's Theorem B \cite{q}*{p.\ 89} it is enough to
  check that for every morphism $u \from s_0 \bar{v}_0 \to s_1 \bar{v}_1$
  in $\Frac_{\cat{C}}(e, b)$, the induced functor
  $u_* \from p_* \slice s_0 \bar{v}_0 \to p_* \slice s_1 \bar{v}_1$ is a \whe{}.
  Indeed, $u_*$ has a right adjoint given by pullback provided that
  $u$ is a fibration.

  On the other hand, given a general $u \from e'_0 \to e'_1$,
  we take a factorization $e'_0 \weto \bullet \fto e'_0 \pull_{e \times b} e'_1$
  of $(\id, u)$ in the slice $\cat{C} \fslice e \times b$ which yields a diagram
  \begin{tikzeq*}
  \matrix[diagram]
  {
                & |(t)| e'_0    &                      \\
    |(e1)| e'_1 & |(b)| \bullet & |(e0)| e'_0 \text{.} \\
  };

  \draw[->] (t) to (b);

  \draw[->] (t) to node[above left]  {$u$}   (e1);
  \draw[->] (t) to node[above right] {$\id$} (e0);

  \draw[fib] (b) to node[below] {$\we$} (e1);
  \draw[fib] (b) to node[below] {$\we$} (e0);
  \end{tikzeq*}
  By the previous part of the argument, the two horizontal fibrations induce
  \whe{}s of the respective slices and so does the identity.
  Therefore, $u_*$ is a \whe{} by 2-out-of-3.
\end{proof}

\begin{proposition}\label{hammock-pullback}
  A square
  \begin{tikzeq*}
  \matrix[diagram]
  {
    |(u)| a & |(x)| c \\
    |(v)| b & |(y)| d \\
  };

  \draw[->] (u) to node[left]  {$p$} (v);
  \draw[->] (x) to node[right] {$q$} (y);

  \draw[->] (u) to node[above] {$f$} (x);
  \draw[->] (v) to node[below] {$g$} (y);
  \end{tikzeq*}
  in $\cat{C}$ is a homotopy pullback \iff{} the associated square
  \begin{tikzeq*}
  \matrix[diagram,column sep={8em,between origins}]
  {
    |(u)| \hammrep{\cat{C},a} & |(x)| \hammrep{\cat{C},c} \\
    |(v)| \hammrep{\cat{C},b} & |(y)| \hammrep{\cat{C},d} \\
  };

  \draw[->] (u) to (v);
  \draw[->] (x) to (y);
  \draw[->] (u) to (x);
  \draw[->] (v) to (y);
  \end{tikzeq*}
  is a homotopy pullback of presheaves.
\end{proposition}

\begin{proof}
  The statement is invariant under weak equivalences of squares in $\cat{C}$ so
  we can assume that both $p$ and $q$ are fibrations.
  Then the square in $\cat{C}$ is a homotopy pullback if and only if
  the morphism $(p, f) \from a \to b \pull_d c$ is a weak equivalence while
  \Cref{hpb-fibers,fiber-representables} imply that
  the square of presheaves is a homotopy pullback if and only if
  for all $e \in \cat{C}$ and $s \bar{v} \from e \rfto b$ the induced functor
  $g_* \from p_* \slice s \bar{v} \to q_* \slice g s \bar{v}$ is a \whe{}.
  We need to verify that these conditions are equivalent.

  If the square in $\cat{C}$ is a homotopy pullback we can further assume that
  it is also a strict pullback.
  In this case $p_* \slice s \bar{v} \to q_* \slice g s \bar{v}$ turns out to be
  an isomorphism of categories.
  We construct its inverse as follows.
  Fix an object of $q_* \slice g s \bar{v}$, i.e.,
  \begin{tikzeq*}
  \matrix[diagram,row sep=1em]
  {
             &              &                  &[3em]                                                          &[3em]
             & |(e''1)| e'' &                  \\
    |(e0)| e & |(e''0)| e'' & |(c)| c          & \text{ in} \Frac^\Reedy_{\cat{C}}(e, c) \text{ together with} &
    |(e1)| e &              & |(d)| d \text{.} \\
             &              &                  &                                                               &
             & |(e')| e'    &                  \\
  };

  \draw[->] (e''0) to node[above] {$w$} node[below] {$\we$} (e0);

  \draw[->] (e''0) to node[above] {$t$} (c);

  \draw[->] (e''1) to node[above left]  {$w$}  (e1);
  \draw[->] (e')   to node[below left]  {$v$}  (e1);
  \draw[->] (e''1) to node[right]       {$u$}  (e');
  \draw[->] (e''1) to node[above right] {$qt$} (d);
  \draw[->] (e')   to node[below right] {$gs$} (d);
  \end{tikzeq*}
  By the universal property of the original pullback we obtain
  a morphism $(su, t) \from e'' \to a$.
  This yields
  \begin{tikzeq*}
  \matrix[diagram,row sep=1em]
  {
             &              &         &[3em]                                                          &[3em]
             & |(e''1)| e'' &         \\
    |(e0)| e & |(e''0)| e'' & |(a)| a & \text{ in} \Frac^\Reedy_{\cat{C}}(e, a) \text{ together with} &
    |(e1)| e &              & |(b)| b \\
             &              &         &                                                               &
             & |(e')| e'    &         \\
  };

  \draw[->] (e''0) to node[above] {$w$} node[below] {$\we$} (e0);

  \draw[->] (e''0) to node[above] {$(su, t)$} (a);

  \draw[->] (e''1) to node[above left]  {$w$}        (e1);
  \draw[->] (e')   to node[below left]  {$v$}        (e1);
  \draw[->] (e''1) to node[right]       {$u$}        (e');
  \draw[->] (e''1) to node[above right] {$p(su, t)$} (b);
  \draw[->] (e')   to node[below right] {$s$}        (b);
  \end{tikzeq*}
  which is an object of $p_* \slice s \bar{v}$.
  This defines the inverse of $g_*$.

  Conversely, assume that the square of presheaves is a homotopy pullback.
  The argument above shows that so is
  \begin{tikzeq*}
  \matrix[diagram,column sep={8em,between origins}]
  {
    |(p)| \hammrep{\cat{C},b \pull_d c} & |(x)| \hammrep{\cat{C},c} \\
    |(v)| \hammrep{\cat{C},b}           & |(y)| \hammrep{\cat{C},d} \\
  };

  \draw[->] (p) to (v);
  \draw[->] (x) to (y);
  \draw[->] (p) to (x);
  \draw[->] (v) to (y);
  \end{tikzeq*}
  and hence the map
  $(p, f)_* \from \hammrep{\cat{C}, a} \to \hammrep{\cat{C}, b \pull_d c}$ is
  a weak equivalence.
  Thus so is $(p, f)$ by \Cref{Yoneda-we}, i.e.,
  the original square is a homotopy pullback.
\end{proof}

\begin{theorem}\label{representable-tribe}
  The category $R \cat{C}$ of homotopy representable fibrant presheaves over
  the hammock localization of a fibration category is a tribe.
\end{theorem}

\begin{proof}
  By a retract argument as in the proof of \Cref{tribe-Reedy},
  a map is anodyne in the category of homotopy representable fibrant presheaves
  \iff{} it is anodyne in the category of all fibrant presheaves.
  Thus in light of \Cref{presheaves-tribe,representable-invariant},
  it suffices to verify that the terminal presheaf is homotopy representable and
  that homotopy representable presheaves are closed under homotopy pullbacks.

  The first claim follows directly from \Cref{hammock-terminal}.
  For the second one, by \Cref{hpb-invariant}, it is enough to
  consider a strict pullback square
  \begin{tikzeq*}
  \matrix[diagram]
  {
    |(U)| A & |(X)| C \\
    |(V)| B & |(Y)| D \\
  };

  \draw[->] (U) to (V);
  \draw[->] (X) to (Y);

  \draw[->] (U) to (X);
  \draw[->] (V) to (Y);
  \end{tikzeq*}
  where both $C \to D$ and $B \to D$ are fibrations.
  We assume that $B$, $C$ and $D$ are homotopy representable and
  by \Cref{strict-representation} we can pick
  compatible representing objects $b$, $c$ and $d$ and
  fibrations $c \to d$ and $b \to d$.
  We denote the resulting pullback in $\cat{C}$ by $a$ and obtain a cube
  \begin{tikzeq*}
  \matrix[diagram,column sep={6em,between origins}]
  {
      |(u)| \hammrep{\cat{C},a} & & |(x)| \hammrep{\cat{C},c} & \\
    & |(U)| A                   & & |(X)| C \\
      |(v)| \hammrep{\cat{C},b} & & |(y)| \hammrep{\cat{C},d} & \\
    & |(V)| B                   & & |(Y)| D \\
  };

  \draw[->] (u) to (v);
  \draw[->] (x) to (y);

  \draw[->] (u) to (x);
  \draw[->] (v) to (y);

  \draw[->,over] (U) to (V);
  \draw[->] (X) to (Y);

  \draw[->,over] (U) to (X);
  \draw[->] (V) to (Y);

  \draw[->] (u) to (U);
  \draw[->] (v) to (V);
  \draw[->] (x) to (X);
  \draw[->] (y) to (Y);
  \end{tikzeq*}
  where both the front and the back face are homotopy pullbacks
  (by \Cref{hammock-pullback}).
  It follows from the Gluing Lemma that
  $\hammrep{\cat{C},a} \to A$ is a weak equivalence, i.e.,
  $A$ is homotopy representable.
\end{proof}

The tribe $R \cat{C}$ is in fact (semi)simplicial, but
this enrichment will play no role in our arguments.

We conclude the section with a technical result which
will be needed in \Cref{tribes-of-presheaves}.
Let $F \from \cat{C} \to \cat{D}$ be an exact functor of fibration categories
and let $A$ be a presheaf over $\hamm \cat{C}$.
A \emph{left homotopy Kan extension} of $A$ along $F$ is
a presheaf $X$ over $\hamm \cat{D}$ together with a map $A \to F^* X$ \st{}
there is a projectively cofibrant replacement $A' \weto A$ \st{}
the adjoint transpose $\Lan_F A' \to X$ of the composite $A' \to A \to F^* X$
is a weak equivalence.
(Note that a representation of $A$ is in particular
a projectively cofibrant replacement of $A$.)
Here, we write $\Lan_F$ for the strict left Kan extension along $F$.

\begin{lemma}\label{Lan-pullback}
  Let $F \from \cat{C} \to \cat{D}$ be an exact functor of fibration categories
  and let
  \begin{tikzeq*}
  \matrix[diagram]
  {
    |(U)| X_\emptyset & |(X)| X_1    \\
    |(V)| X_0         & |(Y)| X_{01} \\
  };

  \draw[->] (U) to (V);
  \draw[->] (X) to (Y);

  \draw[->] (U) to (X);
  \draw[->] (V) to (Y);
  \end{tikzeq*}
  be a homotopy pullback of presheaves over $\hamm \cat{D}$ and let
  \begin{tikzeq*}
  \matrix[diagram,column sep={6em,between origins}]
  {
      |(A)| A_\emptyset     & & |(C)| A_1        & \\
    & |(U)| F^* X_\emptyset & & |(X)| F^* X_1      \\
      |(B)| A_0             & & |(D)| A_{01}     & \\
    & |(V)| F^* X_0         & & |(Y)| F^* X_{01}   \\
  };

  \draw[->] (A) to (B);
  \draw[->] (C) to (D);

  \draw[->] (A) to (C);
  \draw[->] (B) to (D);

  \draw[->,over] (U) to (V);
  \draw[->] (X) to (Y);

  \draw[->,over] (U) to (X);
  \draw[->] (V) to (Y);

  \draw[->] (A) to (U);
  \draw[->] (B) to (V);
  \draw[->] (C) to (X);
  \draw[->] (D) to (Y);
  \end{tikzeq*}
  be a diagram of presheaves over $\hamm \cat{C}$ where
  the back face is also a homotopy pullback and
  $A_0$, $A_1$ and $A_{01}$ are all homotopy representable.
  If three of the diagonal arrows exhibit $X_0$, $X_1$ and $X_{01}$ as
  left homotopy Kan extensions of $A_0$, $A_1$ and $A_{01}$, respectively, then
  the fourth one exhibits
  $X_\emptyset$ as a left homotopy Kan extension of $A_\emptyset$.
\end{lemma}

\begin{proof}
  Without loss of generality we may assume that the back face is
  a strict pullback of two projective fibrations.
  To see that we view the cube as a map of squares $A \to F^* X$ and
  we choose a weak equivalence $X \weto X'$ where
  $X'$ is a pullback of two projective fibrations.
  We then factor the composite $A \to F^* X \to F^* X'$ as
  $A \weto A' \to F^* X'$ with
  $A'$ also a pullback of two projective fibrations.
  We do so by first factoring the map of underlying cospans as
  a weak equivalence followed by a Reedy fibration and
  completing the resulting cospan to a pullback.

  We pick a square $a'$ in $\cat{C}$ together with compatible representations
  $\hammrep{\cat{C}, a'} \weto A'$ as
  in the proof of \Cref{representable-tribe}.
  By the assumption all the adjoint transposes
  $\hammrep{\cat{D}, F a_0} \to X'_0$, $\hammrep{\cat{D}, F a_1} \to X'_1$ and
  $\hammrep{\cat{D}, F a_{01}} \to X'_{01}$ are weak equivalences.
  Thus, by the Gluing Lemma, so is
  $\hammrep{\cat{D}, F a_\emptyset} \to X'_\emptyset$ which concludes the proof.
\end{proof}

%% file: 7-approximation-fibration-categories-by-tribes.tex
\section{Approximating fibration categories by tribes}
\label{tribes-of-presheaves}

In this section we prove the following result.
\begin{theorem}\label{trb-fibcat-App}
  The forgetful functor $\Trb \to \FibCat$ of \Cref{Trb-FibCat} satisfies
  the approximation property (App2).
\end{theorem}
Later, we will employ semisimplicial fibration categories and
semisimplicial tribes to prove that this functor is a DK-equivalence.
Note that this does not follow directly from the theorem above since
$\Trb$ is not known to be a fibration category.

Throughout this section, we fix a fibration category $\cat{C}$,
a tribe $\cat{T}$ and an exact functor $F \from \cat{C} \to \cat{T}$.
We will prove \Cref{trb-fibcat-App} by constructing a diagram
\begin{tikzeq*}
\matrix[diagram]
{
  |(C)|  \cat{C}       & |(T)|  \cat{T}       \\
  |(hC)| \hat{\cat{C}} & |(bC)| \bar{\cat{C}} \\
};

\draw[->] (C)  to node [above] {$F$}   (T);
\draw[->] (hC) to node [below] {$\we$} (bC);
\draw[->] (bC) to                      (T);

\draw[->,shorten <=0.2em] (hC) to node [left]  {$\we$} (C);
\end{tikzeq*}
where $\hat{\cat{C}}$ is a fibration category and $\bar{\cat{C}}$ is a tribe.
Both of these categories are variations of
$R \cat{C}$ of \Cref{representable-tribe}.
Their objects are homotopy representable fibrant presheaves over $\hamm \cat{C}$
equipped with additional structure that ensures
the existence of the functors in the diagram above, in particular,
that $\bar{\cat{C}} \to \cat{T}$ is a homomorphism.

We impose certain cardinality restriction to ensure that
the categories $\bar{\cat{C}}$ and $\hat{\cat{C}}$ are small.
We fix a cardinal number $\kappa$ \st{}
$\hammrep{\cat{C}, a}$ is $\kappa$-small for all $a \in \cat{C}$,
$F^* \hammrep{\cat{T}, x}$ is $\kappa$-small for all $x \in \cat{T}$
and there is an (acyclic cofibration, fibration) factorization functor in
the category of simplicial presheaves over $\hamm \cat{C}$ \st{}
for every map of $\kappa$-small presheaves,
the presheaf resulting from the factorization is also $\kappa$-small.
% \begin{enumerate}
% \item $\hammrep{\cat{C}, a}$ is $\kappa$-small for all $a \in \cat{C}$;
% \item $F^* \hammrep{\cat{T}, x}$ is $\kappa$-small for all $x \in \cat{T}$;
% \item there is an (acyclic cofibration, fibration) factorization functor in
%   the category of simplicial presheaves over $\hamm \cat{C}$ \st{}
%   for every map of $\kappa$-small presheaves,
%   the presheaf resulting from the factorization is also $\kappa$-small.
% \end{enumerate}
Throughout this section all presheaves will be implicitly assumed to
be $\kappa$-small.
Thanks to this assumption,
the categories $\bar{\cat{C}}$ and $\hat{\cat{C}}$ constructed below
will be essentially small and so can be replaced by equivalent small ones.

First, we construct a category $\bar{\cat{C}}$ as follows.
An object is a tuple consisting of
\begin{enumerate}
\item a fibrant presheaf $A$ over $\hamm\cat{C}$;
\item an object $x$ of $\cat{T}$;
\item a fibration $\bar{A} \fto A \times F^* \hammrep{\cat{T}, x}$
\end{enumerate}
subject to the following conditions:
\begin{enumerate}
\setcounter{enumi}{3}
\item the map $\bar{A} \to A$ is a weak equivalence;
\item there is a representation $\hammrep{\cat{C}, a} \weto \bar{A}$ \st{}
  the composite $\hammrep{\cat{C}, a} \to \bar{A} \to F^* \hammrep{\cat{T}, x}$
  corresponds to a weak equivalence
  $\hammrep{\cat{T}, F a} \to \hammrep{\cat{T}, x}$.
\end{enumerate}
Note that these conditions imply that $A$ and $\bar{A}$ are homotopy representable,
but not that $\bar{A} \to A$ is a fibration nor that $\bar{A}$ is fibrant.
We will denote such an object by $(A, \bar{A}, x)$
suppressing the structure map.

A morphism $(A, \bar{A}, x) \to (B, \bar{B}, y)$ consists of
maps of presheaves $A \to B$ and $\bar{A} \to \bar{B}$ and
a morphism $x \to y$ in $\cat{T}$ that are compatible in the sense that
the square
\begin{tikzeq*}
\matrix[diagram,column sep={12em,between origins}]
{
  |(bA)| \bar{A}                           & |(bB)| \bar{B}                           \\
  |(Ax)| A \times F^* \hammrep{\cat{T}, x} & |(By)| B \times F^* \hammrep{\cat{T}, y} \\
};

\draw[->] (bA) to (bB);
\draw[->] (Ax) to (By);

\draw[->,shorten >=0.2em] (bA) to (Ax);
\draw[->,shorten >=0.2em] (bB) to (By);
\end{tikzeq*}
commutes.
We call such a morphism a \emph{fibration} if
\begin{enumerate}
\item $A \to B$ is a fibration of presheaves;
\item $x \to y$ is a fibration in $\cat{T}$;
\item the map $\bar{A} \to
  \bar{B} \pull_{B \times F^* \hammrep{\cat{T}, y}} (A \times F^* \hammrep{\cat{T}, x})$
  induced from the square above is a fibration of presheaves.
\end{enumerate}
Note that this definition does not imply that
$\bar{A} \to \bar{B}$ is a fibration.

The category $\bar{\cat{C}}$ could be viewed as
a variant of the gluing construction in the sense of \cite{sh}*{Sec.~13}.
However, it does not fall under the scope of this construction since
the categories involved are not tribes (or even fibration categories) and
the functor between them is not exact.

We proceed to prove that $\bar{\cat{C}}$ is a tribe and to
characterize its anodyne morphisms (\Cref{C-bar-tribe}) and
weak equivalences (\Cref{C-bar-levelwise}).

\begin{lemma}\label{C-bar-limits}
  In the category $\bar{\cat{C}}$:
  \begin{enumerate}
  \item the object $(1, F^* \hammrep{\cat{T}, 1}, 1)$ with the identity structure map
    is terminal;
  \item pullbacks along fibrations exist.
  \end{enumerate}
\end{lemma}

\begin{proof}
  Part (1) is immediate.

  For part (2), given a fibration $(A, \bar{A}, x) \fto (B, \bar{B}, y)$ and
  any morphism $(C, \bar{C}, z) \to (B, \bar{B}, y)$
  we construct a pullback as follows.
  First, form pullbacks
  \begin{tikzeq*}
  \matrix[diagram]
  {
    |(p)| z \pull_y x & |(x)| x &[2em] |(P)| C \pull_B A & |(A)| A \\
    |(z)| z           & |(y)| y &      |(C)| C           & |(B)| B \\
  };

  \draw[->]  (p) to (x);
  \draw[->]  (z) to (y);
  \draw[->]  (p) to (z);
  \draw[fib] (x) to (y);

  \draw[->]  (P) to (A);
  \draw[->]  (C) to (B);
  \draw[->]  (P) to (C);
  \draw[fib] (A) to (B);
  \end{tikzeq*}
  and then construct a diagram
  \begin{tikzeq*}
  \matrix[diagram,column sep={9em,between origins}]
  {
    |(bP)| \bar{C \pull_B A} &                & |(bA)| \bar{A} &                \\[3ex]
    |(bl)| \bullet           &                & |(br)| \bullet &                \\
                             & |(bC)| \bar{C} &                & |(bB)| \bar{B} \\
      |(Pp)| (C \pull_B A) \times F^* \hammrep{\cat{T}, z \pull_y x} &&
      |(Ax)| A \times F^* \hammrep{\cat{T}, x} & \\
    & |(Cz)| C \times F^* \hammrep{\cat{T}, z} &&
      |(By)| B \times F^* \hammrep{\cat{T}, y} \\
  };

  \draw[->] (bP) to (bA);
  \draw[->] (bl) to (br);
  \draw[->] (Pp) to (Ax);
  % \draw[->] (bC) to (bB);
  \draw[->] (Cz) to (By);

  \draw[->]  (bP) to (bl);
  \draw[->]  (bl) to (Pp);
  \draw[fib] (bA) to (br);

  \draw[fib,over,shorten >=0.2em] (bC) to (Cz);
  \draw[->,shorten >=0.2em]       (br) to (Ax);
  \draw[fib,shorten >=0.2em]      (bB) to (By);

  \draw[->,over] (bC) to (bB);
  
  \draw[->] (bl) to (bC);
  \draw[->] (br) to (bB);
  \draw[->] (Pp) to (Cz);
  \draw[->] (Ax) to (By);
  \end{tikzeq*}
  as follows.
  The bottom face of the cube is obtained by combining the two squares above.
  Then the left and right face are formed by taking pullbacks which gives rise to
  the two objects denoted by bullets.
  The right one has an induced map from $\bar{A}$ and
  the square at the very top is constructed by taking a pullback again.

  Next, we show that $(C \pull_B A, \bar{C \pull_B A}, z \pull_y x)$ is
  an object of $\bar{\cat{C}}$.
  By construction the bottom, left and right faces of the cube above as well as
  the square at the very top are homotopy pullbacks.
  Thus so is the top face of the cube below.
  \begin{tikzeq*}
  \matrix[diagram,column sep={6em,between origins}]
  {
    |(bP)| \bar{C \pull_B A} &                & |(bA)| \bar{A} &                \\\
                             & |(bC)| \bar{C} &                & |(bB)| \bar{B} \\
      |(Pp)| C \pull_B A &&
      |(Ax)| A  & \\
    & |(Cz)| C  &&
      |(By)| B  \\
  };

  % \draw[->] (bP) to (bA);
  \draw[->] (bP) to (bA);
  \draw[->] (Pp) to (Ax);
  % \draw[->] (bC) to (bB);
  \draw[->] (Cz) to (By);

  % \draw[->]      (bP) to (bl);
  \draw[->]      (bP) to (Pp);
  \draw[->,over] (bC) to (Cz);
  % \draw[->]      (bA) to (br);
  \draw[->]      (bA) to (Ax);
  \draw[->]      (bB) to (By);

  \draw[->,over] (bC) to (bB);
  
  \draw[->] (bP) to (bC);
  \draw[->] (bA) to (bB);
  \draw[->] (Pp) to (Cz);
  \draw[->] (Ax) to (By);
  \end{tikzeq*}
  The bottom face is also a homotopy pullback.
  Since three of the vertical arrows are weak equivalences,
  so is $\bar{C \pull_B A} \to C \pull_B A$.
  Moreover, by \Cref{Lan-pullback}, $\bar{C \pull_B A}$ has
  a representation $\hammrep{\cat{C}, d} \weto \bar{C \pull_B A}$ \st{}
  the adjoint transpose of the composite
  $\hammrep{\cat{C}, d} \weto \bar{C \pull_B A} \to F^* \hammrep{\cat{T}, z \pull_y x}$
  is a weak equivalence.
  % Thus $(C \pull_B A, \bar{C \pull_B A}, z \pull_y x)$ is an object of $\bar{\cat{C}}$.

  The universal properties of the pullbacks constructed above imply that
  $(C \pull_B A, \bar{C \pull_B A}, z \pull_y x)$ is
  a pullback of the two original morphisms in $\bar{\cat{C}}$.
  (Even though $\bar{C \pull_B A}$ is
  not the same as $\bar{C} \pull_{\bar B} \bar{A}$.)
\end{proof}

\begin{lemma}\label{C-bar-tribe}
  The category $\bar{\cat{C}}$ with fibrations as defined above
  is a tribe.
  A morphism $(A, \bar{A}, x) \to (B, \bar{B}, y)$ is anodyne \iff{}
  all $A \to B$, $\bar{A} \to \bar{B}$ and $x \to y$ are.
  Moreover, the forgetful functor $\bar{\cat{C}} \to \cat{T}$ is
  a homomorphism of tribes.
\end{lemma}

In the statement of this proposition (as well as in the proof below),
it is a slight abuse of language to
call the map $\bar{A} \to \bar{B}$ ``anodyne'' since
it is not a morphism of a tribe ($\bar{A}$ and $\bar{B}$ are not fibrant).
However,
acyclic cofibrations of the injective model structure on simplicial presheaves
enjoy all the necessary properties of anodyne morphisms in a tribe and
the proof applies as written.

\begin{proof}
  \Cref{tribe-terminal,tribe-pullback} follow by \Cref{C-bar-limits}.

  \Cref{tribe-factorization} is proven by an argument similar to the proof of
  \Cref{tribe-Reedy}.
  That is, every morphism can be factored as
  a levelwise anodyne morphism followed by a fibration, and then
  a retract argument shows that every anodyne morphism is levelwise anodyne.
  Conversely, every levelwise anodyne morphism is anodyne.
  Thus every morphism factors as an anodyne morphism followed by a fibration.

  Every anodyne morphism $z \acto y$ in $\cat{T}$ has a retraction and so
  it induces
  an acyclic cofibration $\hammrep{\cat{T}, y} \to \hammrep{\cat{T}, z}$ by
  \Cref{anodyne-he,Yoneda-we}.
  Thus by the construction of pullbacks in $\bar{\cat{C}}$,
  levelwise anodyne morphisms are stable under pullbacks along fibrations and
  hence so are the anodyne morphisms which proves \Cref{tribe-anodyne}.

  The functor $\bar{\cat{C}} \to \cat{T}$ is a homomorphism since
  anodyne morphisms in $\bar{\cat{C}}$ are levelwise.
\end{proof}

\begin{lemma}\label{C-bar-levelwise}
  A morphism $(A, \bar{A}, x) \to (B, \bar{B}, y)$ in $\bar{\cat{C}}$ is
  a weak equivalence \iff{}
  all $A \to B$ and $\bar{A} \to \bar{B}$ and $x \to y$ are weak equivalences.
\end{lemma}

\begin{proof}
  By an argument similar to the proof of the previous proposition,
  there is a fibration category $\bar{\cat{C}}_{\mathrm{lvl}}$ with
  the same underlying category and fibrations as $\bar{\cat{C}}$ and
  with levelwise weak equivalences.
  We can prove that every object of $\bar{\cat{C}}_{\mathrm{lvl}}$ is cofibrant
  and that path objects in $\bar{\cat{C}}$ and $\bar{\cat{C}}_{\mathrm{lvl}}$ agree
  using the same reasoning as in the proof of \Cref{pullback-of-tribes}.
  Hence the conclusion follows as in the proof of \Cref{tribe-Reedy}.
\end{proof}

Next, we construct the fibration category $\hat{\cat{C}}$.
An object is a tuple consisting of
\begin{enumerate}
\item a fibrant presheaf $A$ over $\hamm\cat{C}$;
\item an object $a$ of $\cat{C}$;
\item a fibration $\bar{A} \fto A \times F^* \hammrep{\cat{T}, F a}$;
\item a representation $\hammrep{\cat{C}, a} \weto \bar{A}$
\end{enumerate}
subject to the following conditions:
\begin{enumerate}
\setcounter{enumi}{4}
\item the map $\bar{A} \to A$ is a weak equivalence;
\item the composite
  $\hammrep{\cat{C}, a} \to \bar{A} \to F^* \hammrep{\cat{T}, F a}$ is
  the unit of the adjunction $\Lan_F \adj F^*$.
\end{enumerate}
Such an object will be denoted by $(A, \bar{A}, a)$.

A morphism $(A, \bar{A}, a) \to (B, \bar{B}, b)$ consists of
maps of presheaves $A \to B$ and $\bar{A} \to \bar{B}$ and
a morphism $a \to b$ in $\cat{C}$ that are compatible in the sense that
the diagram
\begin{tikzeq*}
\matrix[diagram,column sep={12em,between origins}]
{
  |(a)|  \hammrep{\cat{C}, a}                & |(b)|  \hammrep{\cat{C}, b}                \\
  |(bA)| \bar{A}                             & |(bB)| \bar{B}                             \\
  |(Aa)| A \times F^* \hammrep{\cat{T}, F a} & |(Bb)| B \times F^* \hammrep{\cat{T}, F b} \\
};

\draw[->] (a)  to (b);
\draw[->] (bA) to (bB);
\draw[->] (Aa) to (Bb);

\draw[->] (a)  to (bA);
\draw[->] (b)  to (bB);

\draw[->,shorten >=0.2em] (bA) to (Aa);
\draw[->,shorten >=0.2em] (bB) to (Bb);
\end{tikzeq*}
commutes.
We call such a morphism a weak equivalence if all $A \to B$,
$\bar{A} \to \bar{B}$ and $a \to b$ are weak equivalences.
We call it a \emph{fibration} if
\begin{enumerate}
\item $A \to B$ is a fibration of presheaves;
\item $a \to b$ is a fibration in $\cat{C}$;
\item the induced map $\bar{A} \to
  \bar{B} \pull_{B \times F^* \hammrep{\cat{T}, F b}} (A \times F^* \hammrep{\cat{T}, F a})$
  induced from the square above is a fibration of presheaves.
\end{enumerate}

\begin{lemma}\label{C-hat-fibcat}
  The category $\hat{\cat{C}}$ with weak equivalences and fibrations
  as defined above is a fibration category.
  Moreover, the forgetful functors $\hat{\cat{C}} \to \cat{C}$ and
  $\hat{\cat{C}} \to \bar{\cat{C}}$ are exact.
\end{lemma}

\begin{proof}
  Axioms \Cref{fibcat-terminal,fibcat-pullback} follow by
  an argument similar to the one in the proof of \Cref{C-bar-tribe} and
  \Cref{fibcat-2-out-of-6} is immediate.
  % Every morphism can be factored as
  % a levelwise weak equivalence followed by a fibration by
  % an argument similar to the one in the proof of \Cref{tribe-cosieve-fib}.

  To factor a morphism $(A, \bar{A}, a) \to (B, \bar{B}, b)$ as
  a weak equivalence followed by a fibration we proceed as follows.
  First, we factor $A \to B$ as $A \weto A' \fto B$ and
  $a \to a \times b$ as $a \weto a' \fto a \times b$ which
  results in a factorization $a \weto a' \fto b$ where the first morphism is
  not only a weak equivalence but also a split monomorphism.
  Thus the induced map $\hammrep{\cat{C}, a} \to \hammrep{\cat{C}, a'}$ is
  an acyclic injective cofibration and so the lifting problem
  \begin{tikzeq*}
  \matrix[diagram,column sep={8em,between origins}]
  {
    |(a)|  \hammrep{\cat{C}, a}  & |(A)| A                    & |(A')| A' \\
    |(a')| \hammrep{\cat{C}, a'} & |(b)| \hammrep{\cat{C}, b} & |(B)| B \\
  };

  \draw[cof] (a) to node[left]  {$\we$} (a');

  \draw[fib] (A') to (B);

  \draw[->] (a)  to (A);
  \draw[->] (A)  to (A');
  \draw[->] (a') to (b);
  \draw[->] (b)  to (B);
  \end{tikzeq*}
  has a solution $\hammrep{\cat{C}, a'} \to A'$.
  This induces a morphism
  \begin{equation*}
    \bar{A} \push_{\hammrep{\cat{C}, a}} \hammrep{\cat{C}, a'} \to
    \bar{B} \pull_{B \times F^* \hammrep{\cat{T}, F b}} (A \times F^* \hammrep{\cat{T}, F a'})
  \end{equation*}
  which we factor as
  \begin{equation*}
    \bar{A} \push_{\hammrep{\cat{C}, a}} \hammrep{\cat{C}, a'} \weto \bar{A}' \fto
    \bar{B} \pull_{B \times F^* \hammrep{\cat{T}, F b}} (A \times F^* \hammrep{\cat{T}, F a'})
    \text{.}
  \end{equation*}
  This yields the required factorization
  $(A, \bar{A}, a) \weto (A', \bar{A}', a') \fto (B, \bar{B}, b)$.
  (Note that the map $\bar{A} \to \bar{A} \push_{\hammrep{\cat{C}, a}} \hammrep{\cat{C}, a'}$
  is a weak equivalence as a pushout of
  an acyclic cofibration $\hammrep{\cat{C}, a} \to \hammrep{\cat{C}, a'}$ and thus
  the map $(A, \bar{A}, a) \to (A', \bar{A}', a')$ is indeed a weak equivalence.)

  The functors $\hat{\cat{C}} \to \cat{C}$ and $\hat{\cat{C}} \to \bar{\cat{C}}$
  are exact by construction, with the latter
  preserving weak equivalences by \Cref{C-bar-levelwise}.
\end{proof}

Finally, in the three remaining lemmas we prove that
the functors $\hat{\cat{C}} \to \cat{C}$ and $\bar{\cat{C}} \to \cat{T}$ are
weak equivalences.

\begin{lemma}\label{C-hat-C-equivalence}
  The functor $\hat{\cat{C}} \to \cat{C}$ is a weak equivalence.
\end{lemma}

\begin{proof}
  We construct a functor $\cat{C} \to \hat{\cat{C}}$
  sending $a$ to $(\Phi a, \bar{\Phi} a, a)$ as follows.
  Given $a \in \cat{C}$, take $\Phi a$ to be
  the functorial fibrant replacement of $\hammrep{\cat{C}, a}$.
  Then (functorially) factor
  $\hammrep{\cat{C}, a} \to \Phi a \times F^* \hammrep{\cat{T}, F a}$
  as a weak equivalence $\hammrep{\cat{C}, a} \weto \bar{\Phi} a$ followed by
  a fibration $\bar{\Phi} a \fto \Phi a \times F^* \hammrep{\cat{T}, F a}$.
  Clearly, the composite $\cat{C} \to \hat{\cat{C}} \to \cat{C}$ is the identity.

  We will show that the composite $\hat{\cat{C}} \to \cat{C} \to \hat{\cat{C}}$
  is weakly equivalent to the identity, i.e., that
  every object $(A, \bar{A}, a) \in \hat{\cat{C}}$ can be connected to
  $(\Phi a, \bar{\Phi} a, a)$ by a zig-zag of natural weak equivalences.
  First, factor $\hammrep{\cat{C}, a} \to A \times \Phi a$ as
  a weak equivalence $\hammrep{\cat{C}, a} \weto \Psi(A, \bar{A}, a)$
  followed by a fibration $\Psi(A, \bar{A}, a) \fto A \times \Phi a$.
  Next, form a diagram
  \begin{tikzeq*}
  \matrix[diagram]
  {
    |(a)|  \hammrep{\cat{C}, a}        &[12em]                                                                                             \\
    |(bP)| \bar{\Psi}(A, \bar{A}, a)   &                                                                                                   \\
    |(b)|  \bullet                     &       |(Pa)|  \Psi(A, \bar{A}, a) \times F^* \hammrep{\cat{T}, F a}                               \\
    |(bA)| \bar{A} \times \bar{\Phi} a &       |(Aaa)| A \times F^* \hammrep{\cat{T}, F a} \times \Phi a \times F^* \hammrep{\cat{T}, F a} \\
  };

  \draw[->] (a) to[bend right=50] (bA);
  \draw[->] (a) to[bend left]     (Pa);

  \draw[fib] (bA) to (Aaa);
  \draw[->]  (Pa) to (Aaa);
  \draw[->]  (b)  to (bA);
  \draw[fib] (b)  to (Pa);

  \draw[->] (a) to node[right] {$\we$} (bP);
  
  \draw[fib] (bP) to (b);
  \end{tikzeq*}
  as follows.
  The right vertical arrow is obtained by combining
  the fibration from the factorization above with
  the diagonal map of $F^* \hammrep{\cat{T}, F a}$.
  The object denoted by a bullet arises by taking a pullback and
  $\bar{\Psi}(A, \bar{A}, a)$ comes from factoring the resulting map as
  a weak equivalence followed by a fibration.
  Then $(\Psi(A, \bar{A}, a), \bar{\Psi}(A, \bar{A}, a), a)$ is an object of
  $\hat{\cat{C}}$ and the vertical morphisms assemble into weak equivalences
  \begin{tikzeq*}
  \matrix[diagram,column sep={12em,between origins}]
  {
    |(A)|   (A, \bar{A}, a)                                     &
    |(Psi)| (\Psi(A, \bar{A}, a), \bar{\Psi}(A, \bar{A}, a), a) &
    |(Phi)| (\Phi a, \bar{\Phi} a, a) \text{.}                  \\
  };

  \draw[->] (Psi) to node[above] {$\we$} (A);
  \draw[->] (Psi) to node[above] {$\we$} (Phi);

  \node [anchor=east,overlay,inner sep=0, outer sep=0] at (Phi -| 0.5\textwidth,0) {\qedhere};
  \end{tikzeq*}%\qedhere
\end{proof}

\begin{lemma}\label{C-hat-RC-equivalence}
  The forgetful functor $\hat{\cat{C}} \to R \cat{C}$ is a weak equivalence.
\end{lemma}

\begin{proof}
  By \Cref{App}, it suffices to verify the approximation properties.
  (App1) is immediate.
  For (App2) consider $(B, \bar{B}, b)$ in $\hat{\cat{C}}$ and a map $A \to B$.
  Factor it as a weak equivalence $A \weto A'$ followed by
  a fibration $A' \fto B$.
  Since $A'$ is homotopy representable we can pick a morphism $a' \to b$ in $\cat{C}$ and
  a square
  \begin{tikzeq*}
  \matrix[diagram,column sep={8em,between origins}]
  {
    |(a')| \hammrep{\cat{C}, a'} & |(b)| \hammrep{\cat{C}, b} \\
    |(A')| A'                    & |(B)| B                    \\
  };

  \draw[->] (a') to node[left]  {$\we$} (A');
  \draw[->] (b)  to node[right] {$\we$} (B);

  \draw[->]  (a') to (b);
  \draw[fib] (A') to (B);
  \end{tikzeq*}
  by \Cref{strict-representation}.
  By the naturality of the unit we obtain a square
  \begin{tikzeq*}
  \matrix[diagram,column sep={12em,between origins}]
  {
    |(a')| \hammrep{\cat{C}, a'}                 & |(b)| \hammrep{\cat{C}, b}                \\
    |(A')| A' \times F^* \hammrep{\cat{T}, F a'} & |(B)| B \times F^* \hammrep{\cat{T}, F b} \\
  };

  \draw[->,shorten >=0.2em] (a') to (A');
  \draw[->,shorten >=0.2em] (b)  to (B);

  \draw[->] (a') to (b);
  \draw[->] (A') to (B);
  \end{tikzeq*}
  This gives a map from $\hammrep{\cat{C}, a'}$ to the pullback denoted by
  the bullet in the diagram
  \begin{tikzeq*}
  \matrix[diagram,column sep={12em,between origins}]
  {
    |(a')|  \hammrep{\cat{C}, a'}                 & |(b)|  \hammrep{\cat{C}, b}                \\
    |(bA')| \bar{A}'                              & |(bB)| \bar{B}                             \\
    |(bt)|  \bullet                               &                                            \\
    |(A')|  A' \times F^* \hammrep{\cat{T}, F a'} & |(B)|  B \times F^* \hammrep{\cat{T}, F b} \\
  };

  \draw[->] (a') to node[right] {$\we$} (bA');
  \draw[->] (b)  to node[right] {$\we$} (bB);

  \draw[->] (a') to (b);
  \draw[->] (bt) to (bB);
  \draw[->] (A') to (B);

  \draw[->,shorten >=0.2em] (a') to [bend right] (A');

  \draw[fib] (bA') to (bt);

  \draw[fib,shorten >=0.2em] (bt)  to (A');
  \draw[fib,shorten >=0.2em] (bB)  to (B);
  \end{tikzeq*}
  which we then factor as a weak equivalence followed by a fibration.
  Altogether we obtain an object $(A', \bar{A}', a')$ and
  a morphism $(A', \bar{A}', a') \to (B, \bar{B}, b)$ thus completing the proof.
\end{proof}

\begin{lemma}\label{C-bar-RC-equivalence}
  The forgetful functor $\bar{\cat{C}} \to R \cat{C}$ is a weak equivalence.
\end{lemma}

\begin{proof}
  By \Cref{App}, it is enough to verify the approximation properties.
  (App1) is immediate.
  For (App2) consider $(B, \bar{B}, y)$ in $\bar{\cat{C}}$ and a map $A \to B$.
  By the definition, there is an object $b$ and
  a representation $\hammrep{\cat{C}, b} \weto \bar{B}$.
  The map $\hammrep{\cat{T}, F b} \to \hammrep{\cat{T}, y}$,
  the adjoint transpose of the composite
  $\hammrep{\cat{C}, b} \to \bar{B} \to F^* \hammrep{\cat{T}, y}$,
  is a weak equivalence.
  It is induced by a zig-zag $F b \zto y$ homotopic to
  a morphism $w \from F b \to y$ by \Cref{tribe-single-arrow} which
  is a weak equivalence by \Cref{Yoneda-we}.
  This homotopy induces a homotopy commutative triangle
  \begin{tikzeq*}
  \matrix[diagram,column sep={8em,between origins}]
  {
    |(b)| \hammrep{\cat{C}, b} &                                \\
    |(B)| \bar{B}              & |(y)| F^* \hammrep{\cat{T}, y} \\
  };

  \draw[->] (b) to node[left] {$\we$} (B);

  \draw[->] (b) to (y);
  \draw[->] (B) to (y);
  \end{tikzeq*}
  in which the diagonal arrow is induced by $w$.
  By \Cref{HEP}, the map $\hammrep{\cat{C}, b} \weto \bar{B}$
  can be replaced by a homotopic one making the triangle commute strictly.

  Factor the map $A \to B$ as a weak equivalence $A \weto A'$ followed by
  a fibration $A' \fto B$.
  Since $A'$ is homotopy representable we can pick
  a morphism $a' \to b$ in $\cat{C}$ and a square
  \begin{tikzeq*}
  \matrix[diagram,column sep={8em,between origins}]
  {
    |(a')| \hammrep{\cat{C}, a'} & |(b)| \hammrep{\cat{C}, b} \\
    |(A')| A'                    & |(B)| B                    \\
  };

  \draw[->] (a') to node[left]  {$\we$} (A');
  \draw[->] (b)  to node[right] {$\we$} (B);

  \draw[->]  (a') to (b);
  \draw[fib] (A') to (B);
  \end{tikzeq*}
  by \Cref{strict-representation} where
  the map $\hammrep{\cat{C}, b} \weto \bar{B}$ is the one constructed in
  the preceding paragraph.
  Thus we obtain a square
  \begin{tikzeq*}
  \matrix[diagram,column sep={12em,between origins}]
  {
    |(a')| \hammrep{\cat{C}, a'}                 & |(b)| \hammrep{\cat{C}, b}              \\
    |(A')| A' \times F^* \hammrep{\cat{T}, F a'} & |(B)| B \times F^* \hammrep{\cat{T}, y} \\
  };

  \draw[->,shorten >=0.2em] (a') to (A');
  \draw[->,shorten >=0.2em] (b)  to (B);

  \draw[->] (a') to (b);
  \draw[->] (A') to (B);
  \end{tikzeq*}
  where the bottom map is induced by the composite $F a' \to F b \to y$.
  This gives a map from $\hammrep{\cat{C}, a'}$ to the pullback denoted by
  the bullet in the diagram
  \begin{tikzeq*}
  \matrix[diagram,column sep={12em,between origins}]
  {
    |(a')|  \hammrep{\cat{C}, a'}                 & |(b)|  \hammrep{\cat{C}, b}              \\
    |(bA')| \bar{A}'                              & |(bB)| \bar{B}                           \\
    |(bt)|  \bullet                               &                                          \\
    |(A')|  A' \times F^* \hammrep{\cat{T}, F a'} & |(B)|  B \times F^* \hammrep{\cat{T}, y} \\
  };

  \draw[->] (a') to node[right] {$\we$} (bA');
  \draw[->] (b)  to node[right] {$\we$} (bB);

  \draw[->] (a') to (b);
  \draw[->] (bt) to (bB);
  \draw[->] (A') to (B);

  \draw[->,shorten >=0.2em] (a') to [bend right] (A');

  \draw[fib] (bA') to (bt);

  \draw[fib,shorten >=0.2em] (bt)  to (A');
  \draw[fib,shorten >=0.2em] (bB)  to (B);
  \end{tikzeq*}
  which we then factor as a weak equivalence followed by a fibration.
  Altogether we obtain an object $(A', \bar{A}', F a')$ and
  a morphism $(A', \bar{A}', F a') \to (B, \bar{B}, y)$ as required.
\end{proof}

\begin{proof}[Proof of \Cref{trb-fibcat-App}.]
  By \Cref{C-hat-RC-equivalence,C-bar-RC-equivalence} the diagonal morphisms
  in the triangle
  \begin{tikzeq*}
  \matrix[diagram]
  {
    |(hC)| \hat{\cat{C}} &                  & |(bC)| \bar{\cat{C}} \\
                         & |(RC)| R \cat{C} &                      \\
  };

  \draw[->] (hC) to (bC);

  \draw[->] (hC) to node[below left]  {$\we$} (RC);
  \draw[->] (bC) to node[below right] {$\we$} (RC);
  \end{tikzeq*}
  are weak equivalences and hence so is $\hat{\cat{C}} \to \bar{\cat{C}}$.
  This along with \Cref{C-hat-C-equivalence} shows that both labeled arrows
  in the diagram
  \begin{tikzeq*}
  \matrix[diagram]
  {
    |(C)|  \cat{C}       & |(T)|  \cat{T}       \\
    |(hC)| \hat{\cat{C}} & |(bC)| \bar{\cat{C}} \\
  };

  \draw[->] (C)  to node [above] {$F$}   (T);
  \draw[->] (hC) to node [below] {$\we$} (bC);
  \draw[->] (bC) to                      (T);

  \draw[->,shorten <=0.2em] (hC) to node [left]  {$\we$} (C);
  \end{tikzeq*}
  are weak equivalences thus completing the proof.
\end{proof}

%% file: 8-equivalence-between-tribes-and-fibration-categories.tex
\section{Equivalence between tribes and fibration categories}
\label{equivalence}

We are now ready to prove our key theorem.

\begin{theorem}\label{trb-fibcat-DK}
  The forgetful functor $\Trb \to \FibCat$ of \Cref{Trb-FibCat} is
  a DK-equivalence.
\end{theorem}

By \Cref{semisimp-DK}, it suffices to show that the forgetful functor
$\sTrb \to \sFibCat$ is a DK-equivalence, which we will do by
verifying the approximation properties.
This can be accomplished by refining the result of the previous section
for which we will need the following two lemmas.
Let
\begin{tikzeq*}
\matrix[diagram]
{
  |(C)|  \cat{C}       & |(T)|  \cat{T}       \\
  |(hC)| \hat{\cat{C}} & |(bC)| \bar{\cat{C}} \\
};

\draw[->] (C)  to node [above] {$F$}   (T);
\draw[->] (hC) to node [below] {$\we$} (bC);
\draw[->] (bC) to                      (T);

\draw[->,shorten <=0.2em] (hC) to node [left]  {$\we$} (C);
\end{tikzeq*}
be categories and functors introduced in the previous section.

\begin{lemma}\label{App-fibrations}
  \leavevmode
  \begin{enumerate}
  \item The functor $\hat{\cat{C}} \to \cat{C}$ is
    a fibration of fibration categories.
  \item The functor $\bar{\cat{C}} \to \cat{T}$ is a fibration of tribes.
  \end{enumerate}
\end{lemma}

Note that these are fibrations between
non-semisimplicial fibration categories (tribes) as defined in
\Cref{fibcat-fib,tribe-fib}.
Using the next lemma we will promote them to fibrations in
$\sFibCat$ and $\sTrb$.

\begin{proof}
  For part (1), the isofibration condition is immediate.
  The lifting property for WF-factorizations is verified just like axiom (F3)
  in the proof of \Cref{C-hat-fibcat} except that
  a part of the factorization is given in advance.
  It remains to check the lifting property for pseudofactorizations.
  Let $(A, \bar{A}, a) \to (B, \bar{B}, b)$ be a morphism in $\hat{\cat{C}}$ and
  let
  \begin{tikzeq*}
  \matrix[diagram]
  {
    |(a)|  a  & |(b)|   b   \\
    |(a')| a' & |(a'')| a'' \\
  };

  \draw[->]  (a)   to (b);
  \draw[fib] (a'') to (b);

  \draw[->]  (a') to node[below] {$\we$} (a'');
  \draw[fib] (a') to node[left]  {$\we$} (a);
  \end{tikzeq*}
  be a pseudofactorization of its image in $\cat{C}$.
  We form a diagram
  \begin{tikzeq*}
  \matrix[diagram,column sep={12em,between origins}]
  {
    |(a')|  \hammrep{\cat{C}, a'}                & |(b)|  \hammrep{\cat{C}, a}                \\
    |(bA')| \bar{A}'                             & |(bB)| \bar{A}                             \\
    |(bt)|  \bullet                              &                                            \\
    |(A')|  A \times F^* \hammrep{\cat{T}, F a'} & |(B)|  A \times F^* \hammrep{\cat{T}, F a} \\
  };

  \draw[->] (a') to node[right] {$\we$} (bA');
  \draw[->] (b)  to node[right] {$\we$} (bB);

  \draw[->] (a') to (b);
  \draw[->] (bt) to (bB);
  \draw[->] (A') to (B);

  \draw[->,shorten >=0.2em] (a') to [bend right] (A');

  \draw[fib] (bA') to (bt);

  \draw[fib,shorten >=0.2em] (bt)  to (A');
  \draw[fib,shorten >=0.2em] (bB)  to (B);
  \end{tikzeq*}
  by first taking a pullback, denoted by a bullet, and then
  factoring the resulting map as a weak equivalence followed by a fibration.
  This way we obtain
  an acyclic fibration $(A', \bar{A}', a') \afto (A, \bar{A}, a)$.
  To construct the remaining part we lift the factorization of the composite
  $(A', \bar{A}', a') \afto (A, \bar{A}, a) \to (B, \bar{B}, b)$.

  For part (2), the verification of the first four properties is analogous to
  the proof of part (1).
  Next, we verify the lifting property for lifts.
  Let
  \begin{tikzeq*}
  \matrix[diagram,column sep={7em,between origins}]
  {
    |(A)| (A, \bar{A}, u) & |(C)| (C, \bar{C}, x) \\
    |(B)| (B, \bar{B}, v) & |(D)| (D, \bar{D}, y) \\
  };

  \draw[->]  (A) to (C);
  \draw[->]  (B) to (D);
  \draw[fib] (C) to (D);

  \draw[cof] (A) to node[left] {$\we$} (B);
  \end{tikzeq*}
  be a lifting problem in $\hat{\cat{C}}$ and
  fix a solution $v \to x$ of its image in $\cat{C}$.
  Pick any solution $B \to C$ of its image in $R \cat{C}$.
  Since $\bar{A} \to \bar{B}$ is an acyclic cofibration by \Cref{C-bar-tribe},
  there is a lift in
  \begin{tikzeq*}
  \matrix[diagram,column sep={14em,between origins},row sep={5em,between origins}]
  {
    |(A)| \bar{A} & |(C)| \bar{C}                                                                                 \\
    |(B)| \bar{B} & |(M)| \bar{D} \pull_{D \times F^* \hammrep{\cat{T}, y}} (C \times F^* \hammrep{\cat{T}, F x}) \\
  };

  \draw[->]  (A) to (C);
  \draw[->]  (B) to (M);
  \draw[fib] (C) to (M);

  \draw[cof] (A) to node[left] {$\we$} (B);

  \draw[->,dashed] (B) to (C);
  \end{tikzeq*}
  which completes a lift in the original square.
  The proof of the lifting property for cofibrancy lifts is analogous.
\end{proof}

\begin{lemma}\label{Fr-fibrations}
  \leavevmode
  \begin{enumerate}
  \item $\fr \from \FibCat \to \sFibCat$ preserves fibrations.
  \item $\fr \from \Trb \to \sTrb$ preserves fibrations.
  \end{enumerate}
\end{lemma}

\begin{proof}
  Part (1) follows from \cite{s2}*{Lem.\ 1.11(1)}.

  For part (2), consider
  a fibration $P \from \cat{S} \to \cat{T}$ of semisimplicial tribes.
  By part (1), $\fr P$ is a fibration of underlying fibration categories.

  Let $a \to b$ be a morphism in $\fr \cat{S}$ and
  consider a factorization $P a \acto x \fto P b$.
  We lift it to a factorization $a \acto a' \fto b$ inductively.
  First, the factorization $P a_0 \acto x_0 \fto P b_0$ lifts to
  a factorization $a_0 \acto a'_0 \fto b_0$ since $P$ is a fibration.
  For the inductive step, the partial factorization below dimension $m$ induces
  a morphism $a_m \to M_m a' \pull_{M_m b} b_m$ which
  we factor as an anodyne morphism $a_m \acto a'_m$ followed by
  a fibration $a'_m \fto M_m a' \pull_{M_m b} b_m$.
  This proves the lifting property for AF-factorizations.

  The other two conditions are verified in a similar manner.
\end{proof}

\begin{proof}[Proof of \Cref{trb-fibcat-DK}.]
  Consider the square
  \begin{tikzeq*}
  \matrix[diagram,column sep={6em,between origins}]
  {
    |(sT)| \sTrb & |(sF)| \sFibCat \\
    |(T)|  \Trb  & |(F)|  \FibCat  \\
  };

  \draw[->] (sT) to node[left]  {$\we$} (T);
  \draw[->] (sF) to node[right] {$\we$} (F);

  \draw[->] (sT) to (sF);
  \draw[->] (T)  to (F);
  \end{tikzeq*}
  where the vertical functors are DK-equivalences by \Cref{semisimp-DK}.
  The categories $\sTrb$ and $\sFibCat$ are fibration categories
  by \Cref{fibcat-of-tribes} and the functor $\sTrb \to \sFibCat$ is exact
  by \Cref{fibcat-fib,tribe-fib}.
  It suffices to verify that this functor is a DK-equivalence and,
  in light of \Cref{App}, we can do so by checking that
  it satisfies the approximation properties.

  (App1) is immediate.
  For (App2) consider a semisimplicial fibration category $\cat{C}$,
  a semisimplicial tribe $\cat{T}$ and
  a semisimplicial exact functor $F \from \cat{C} \to \cat{T}$.
  By \Cref{trb-fibcat-App} there are a fibration category $\hat{\cat{C}}$,
  a tribe $\bar{\cat{C}}$ (not necessarily semisimplicial) and a square
  \begin{tikzeq*}
  \matrix[diagram]
  {
    |(C)|  \cat{C}       & |(T)|  \cat{T}       \\
    |(hC)| \hat{\cat{C}} & |(bC)| \bar{\cat{C}} \\
  };

  \draw[->] (C)  to node [above] {$F$}   (T);
  \draw[->] (hC) to node [below] {$\we$} (bC);
  \draw[->] (bC) to                      (T);

  \draw[->,shorten <=0.2em] (hC) to node [left]  {$\we$} (C);
  \end{tikzeq*}
  in $\FibCat$.
  We form a diagram
  \begin{tikzeq*}
  \matrix[diagram]
  {
    |(C)| \cat{C} &                        &                           & |(T)| \cat{T} &                        &                           \\
                  & |(hFC)| \frh \cat{C}   &                           &               & |(hFT)| \frh \cat{T}   &                           \\
                  &                        & |(FC)| \fr \cat{C}        &               &                        & |(FT)| \fr \cat{T}        \\
                  & |(hCp)| \hat{\cat{C}}' &                           &               & |(bCp)| \bar{\cat{C}}' &                           \\
                  &                        & |(FhC)| \fr \hat{\cat{C}} &               &                        & |(FbC)| \fr \bar{\cat{C}} \\
  };

  \draw[->] (C) to node[above] {$F$} (T);

  \draw[->] (hCp) to [bend left] (C);
  \draw[->] (bCp) to [bend left] (T);

  \draw[->] (hFC) to node [above right] {$\we$} (C);
  \draw[->] (hFT) to node [above right] {$\we$} (T);

  \draw[->,over] (hFC) to (hFT);
  \draw[->]      (hCp) to (bCp);
  \draw[->]      (hCp) to (hFC);
  \draw[->]      (bCp) to (hFT);

  \draw[->,over]  (FC)  to (FT);
  \draw[fib]      (FbC) to (FT);

  \draw[->]       (FhC) to node[below]       {$\we$} (FbC);
  \draw[fib,over] (FhC) to node[above right] {$\we$} (FC);

  \draw[->] (hFC) to node[above right] {$\we$} (FC);
  \draw[->] (hFT) to node[above right] {$\we$} (FT);

  \draw[->] (hCp) to (FhC);
  \draw[->] (bCp) to (FbC);
  \end{tikzeq*}
  as follows.
  \begin{enumerate}
  \item The front square is obtained by applying $\fr$ to the square above.
  \item The two top squares are naturality squares of
    transformations $\frh \to \fr$ and $\frh \to \id$.
  \item The category $\hat{\cat{C}}'$ is defined as
    the pullback $\frh \cat{C} \pull_{\fr \cat{C}} \fr \hat{\cat{C}}$ which
    can be constructed since $\fr \hat{\cat{C}} \to \fr \cat{C}$ is a fibration
    by \Cref{App-fibrations,Fr-fibrations}.
  \item The category $\bar{\cat{C}}'$ is defined as
    the pullback $\frh \cat{T} \pull_{\fr \cat{T}} \fr \bar{\cat{C}}$ which
    can be constructed since $\fr \bar{\cat{C}} \to \fr \cat{T}$ is a fibration
    by \Cref{App-fibrations,Fr-fibrations}.
  \end{enumerate}
  The top diagonal arrows are weak equivalences
  by \Cref{tribe-frh1,tribe-frh2}.
  The functors $\fr \hat{\cat{C}} \to \fr \cat{C}$ and
  $\fr \hat{\cat{C}} \to \fr \bar{\cat{C}}$ are weak equivalences since
  $\fr$ is homotopical.

  Since the left and right squares are homotopy pullbacks,
  the functors $\hat{\cat{C}}' \to \frh \cat{C}$,
  $\hat{\cat{C}}' \to \fr \hat{\cat{C}}$ and
  $\bar{\cat{C}}' \to \fr \bar{\cat{C}}$ are weak equivalences.
  Therefore, by 2-out-of-3 in the square
  \begin{tikzeq*}
  \matrix[diagram]
  {
    |(C)|  \cat{C}        & |(T)|  \cat{T}        \\
    |(hC)| \hat{\cat{C}}' & |(bC)| \bar{\cat{C}}' \\
  };

  \draw[->] (C)  to node [above] {$F$}   (T);
  \draw[->] (hC) to node [below] {$\we$} (bC);
  \draw[->] (bC) to                      (T);

  \draw[->,shorten <=0.2em] (hC) to node [left] {$\we$} (C);
  \end{tikzeq*}
  the labeled arrows are weak equivalences which completes the proof.
\end{proof}

%% file: 9-application-to-internal-languages.tex
\section{Application to internal languages}
\label{language}

In the final section, we apply our results to establish an equivalence between
categorical models of Martin-L\"of Type Theory with
dependent sums and intensional identity types and
finitely complete $(\infty, 1)$-categories.
We will use comprehension categories as our notion of categorical models.
They were introduced by Jacobs \cite{j-comp} and
developed extensively in \cite{j}.

\begin{definition}
  A \emph{comprehension category} is a category $\ccat{C}$ with a terminal object
  equipped with a Grothendieck fibration $p \from \ccat{T} \to \ccat{C}$ and
  a fully faithful \emph{comprehension} functor
  $\chi \from \ccat{T} \to \ccat{C}^{[1]}$ \st{} the triangle
  \begin{tikzeq*}
    \matrix[diagram]
    {
      |(T)| \ccat{T} &                & |(C1)| \ccat{C}^{[1]} \\
                     & |(C)| \ccat{C} &                       \\
    };

    \draw[->] (T)  to node[above]       {$\chi$} (C1);
    \draw[->] (T)  to node[below left]  {$p$}    (C);
    \draw[->] (C1) to node[below right] {$\cod$} (C);
  \end{tikzeq*}
  commutes and $\chi$ carries cartesian morphisms to pullback squares.
\end{definition}

\begin{definition}
  A \emph{fibration} in a comprehension category $\ccat{C}$ is
  a morphism isomorphic to a composite of morphisms in the image of $\chi$.
\end{definition}

Bare comprehension categories only model the structural rules of
Martin-L\"of Type Theory.
Thus in the definition of a categorical model, we make additional assumptions on
the comprehension category $\ccat{C}$ to ensure that
it carries an interpretation of the type constructors $\tSigma$ and $\Id$.

\begin{definition}
  A \emph{categorical model of type theory} is
  a comprehension category $\ccat{C}$ that has
  \begin{enumerate}
  %\item \emph{strong $\tSigma$-types} in the sense of \cite{lw}*{Def.\ 3.4.4.1};
  \item \emph{$\Sigma$-types} in the sense of \cite{lw}*{Def.~3.4.4.1} with
    strong $\eta$-rule;%
    \footnote{This ensures that the $\Sigma$-type is given by
    composition and hence preserved by morphisms of categorical models of
    type theory, cf.~Def.~9.6.}
  \item \emph{weakly stable $\Id$-types} in the sense of \cite{lw}*{Def.\ 2.3.6};
  \end{enumerate}
  \st{} all objects are fibrant.
\end{definition}

Given $A \in \ccat{T}(\Gamma)$, we will write $\Gamma.A$ for
the domain of $\chi(A)$.
This operation can be extended to dependent contexts as follows.
Given a context $\Delta = (A_1, \ldots, A_m)$ where $A_1 \in \ccat{T}(\Gamma)$,
$A_2 \in \ccat{T}(\Gamma.A_1)$, \ldots,
$A_m \in \ccat{T}(\Gamma.A_1.\cdots.A_{m-1})$, we will write
$\Gamma.\Delta$ for the domain of $\chi(A_m)$.
We will also use Garner's identity contexts \cite{g} which allows us to
form $\Id_\Gamma \in \ccat{T}(\Gamma.\Gamma)$.

\begin{definition}
  Let $\ccat{C}$ be a categorical model.
  \begin{enumerate}
  \item A \emph{homotopy} between morphisms $f, g \from \Gamma \to \Delta$ is
    a morphism $H \from \Gamma \to \Delta.\Delta.\Id_\Delta$ \st{} the triangle
    \begin{tikzeq*}
    \matrix[diagram,column sep={6em,between origins}]
    {
                   & |(Id)| \Delta.\Delta.\Id_\Delta \\
      |(G)| \Gamma & |(DD)| \Delta.\Delta            \\
    };

    \draw[->] (G)  to node[below]      {$f.g$}              (DD);
    \draw[->] (G)  to node[above left] {$H$}                (Id);
    \draw[->] (Id) to node[right]      {$\chi(\Id_\Delta)$} (DD);
    \end{tikzeq*}
    commutes.
  \item A morphism $f \from \Gamma \to \Delta$ is a \emph{homotopy equivalence}
    if there is a morphism $g \from \Delta \to \Gamma$ \st{}
    $f g$ is homotopic to $\id_\Delta$ and $g f$ is homotopic to $\id_\Gamma$.
  \end{enumerate}
\end{definition}

\begin{remark}\label{Sigma}
  Given $\Gamma \in \ccat{C}$, $A \in \ccat{T}(\Gamma)$ and
  $B \in \ccat{T}(\Gamma.A)$, let $\tSigma_A B \in \ccat{T}(\Gamma)$ denote
  the strong $\tSigma$-type of $A$ and $B$.
  For fixed $\Gamma$ and $A$ as above, the assignment $B \mapsto \tSigma_A B$ is
  a left adjoint of the pullback functor
  $\chi(A)^* \from \ccat{T}(\Gamma) \to \ccat{T}(\Gamma.A)$.
  Conversely, if such a left adjoint exists, then
  its values are strong $\tSigma$-types.
\end{remark}

\begin{definition}
  \leavevmode
  \begin{enumerate}
  \item A \emph{morphism between categorical models} $\ccat{C}$ and $\ccat{C}'$
    is a pair of functors $F_0 \from \ccat{C} \to \ccat{C}'$ and
    $F_1 \from \ccat{T} \to \ccat{T}'$ strictly compatible with $p$ and $\chi$
    \st{} $F_0$ preserves a terminal object, $\tSigma$-types and $\Id$-types.
  \item A \emph{weak equivalence of categorical models} is
    a morphism \st{} $F_0$ induces an equivalence of the homotopy categories,
    i.e., the localizations \wrt{} homotopy equivalences.
\end{enumerate}

\end{definition}

The homotopical category of categorical models is
denoted by $\CompCat_{\Id, \tSigma}$.
We will prove that it is DK-equivalent to the category of tribes.

\begin{proposition}\label{compcat-to-trb}
  A categorical model with its subcategory of fibrations is a tribe.
  Moreover, a morphism of categorical models is a homomorphism of tribes.
  This defines a homotopical functor $T \from \CompCat_{\Id, \tSigma} \to \Trb$.
\end{proposition}

\begin{proof}
  Axiom \Cref{tribe-terminal} is satisfied by the assumption while
  \Cref{tribe-pullback} follows from the fact that
  $\chi$ carries cartesian morphisms to pullback squares.
  A factorization of a morphism $f \from \Gamma \to \Delta$ is given by
  \begin{tikzeq*}
  \matrix[diagram,column sep={7em,between origins}]
  {
    |(G)| \Gamma & |(Id)| \Delta.\Gamma.(\id.f)^*\Id_\Delta & |(D)| \Delta \\
  };

  \draw[->] (G)  to (Id);
  \draw[->] (Id) to (D);
  \end{tikzeq*}
  as constructed in
  \cite{gg}*{Lem.\ 11} which proves \Cref{tribe-factorization}.
  Finally, \Cref{tribe-anodyne} follows by \cite{gg}*{Prop.\ 14}.

  A morphism of categorical models preserves fibrations by definition.
  On the other hand, the anodyne morphisms can be characterized as
  those admitting deformation retractions by \cite{gg}*{Lem.\ 13(i)} so
  the conclusion follows by preservation of $\Id$-types.
\end{proof}

Given a tribe $\cat{T}$ we define a category $\cat{T}^{[1]}_{\mathrm{f}}$ as
the full subcategory of $\cat{T}^{[1]}$ spanned by fibrations.

\begin{proposition}\label{trb-to-compcat}
  If $\cat{T}$ is a tribe, then let $\chi$ denote
  the inclusion $\cat{T}^{[1]}_{\mathrm{f}} \ito \cat{T}^{[1]}$.
  Moreover, let $p \from \cat{T}^{[1]}_{\mathrm{f}} \to \cat{T}$ be
  the composite $\cod \chi$.
  Then $(\cat{T}, \cat{T}^{[1]}_{\mathrm{f}}, \chi)$ is a categorical model.
  Moreover, a homomorphism of tribes induces
  a morphism of the associated categorical models,
  yielding a homotopical functor $C \from \Trb \to \CompCat$.
\end{proposition}

\begin{proof}
  The category $\cat{T}$ has a terminal object by assumption.
  The functor $p$ is a Grothendieck fibration and
  $\chi$ preserves cartesian morphisms since
  $p$-cartesian morphisms are exactly pullbacks along fibrations.

  For every fibration $q \from x \fto y$, the pullback functor
  $q^* \from \cat{T}^{[1]}_{\mathrm{f}}(y) \to \cat{T}^{[1]}_{\mathrm{f}}(x)$
  has a left adjoint given by composition and hence by \Cref{Sigma}
  strong $\tSigma$-types exist.

  Moreover, for every fibration $q \from x \fto y$, we choose a factorization
  \begin{tikzeq*}
  \matrix[diagram,column sep={6em,between origins}]
  {
            & |(Id)|  \Id_x                \\
    |(x)| x & |(xyx)| x \pull_y x \text{.} \\
  };

  \draw[->]  (x)  to (xyx);
  \draw[fib] (Id) to (xyx);

  \draw[cof] (x) to node[above left] {$\mathsf{r}_x$} node[below right] {$\we$} (Id);
  \end{tikzeq*}
  To see that $(\Id_x, \mathsf{r}_x)$ is an $\Id$-type, we need to verify that
  for every morphism $f \from y' \to y$,
  the morphism $f^* \mathsf{r}_x \from f^* x \to f^* \Id_x$ is anodyne.
  Indeed, this follows from \Cref{anodyne-pullback}.

  The verification that a homomorphism of tribes induces
  a morphism of categorical model is straightforward.
\end{proof}

\begin{theorem}\label{compcat-trb}
  The functor $T \from \CompCat_{\Id, \tSigma} \to \Trb$
  of \Cref{compcat-to-trb} is a DK-equivalence.
\end{theorem}

\begin{proof}
  We will show that the functor $C$ of \Cref{trb-to-compcat} is
  a homotopy inverse of $T$.
  Clearly, $T C = \id_\Trb$.
  Given a categorical model $(\ccat{C}, \ccat{T}, \chi)$, we construct
  a natural morphism $(F_0, F_1) \from \ccat{C} \to C T \ccat{C}$.
  First, we set $F_0 = \id_{\ccat{C}}$.
  Moreover, $\chi \from \ccat{T} \to \ccat{C}^{[1]}$ factors as
  $F_1 \from \ccat{T} \to \ccat{C}^{[1]}_{\mathrm{f}}$ followed by
  $\ccat{C}^{[1]}_{\mathrm{f}} \ito \ccat{C}^{[1]}$.
  Since homotopy equivalences in $\ccat{C}$ and $T \ccat{C}$ agree,
  this morphism is a weak equivalence.
\end{proof}

Finally, we prove our main theorem.

\begin{theorem}\label{main-theorem}
  The homotopical category of categorical models of Martin-L\"of Type Theory
  with dependent sums and intensional identity types is DK-equivalent to
  the homotopical category of finitely complete $(\infty, 1)$-categories.
  In particular, the associated $(\infty, 1)$-categories are equivalnent.
\end{theorem}

\begin{proof}
  We consider the composite
  \begin{tikzeq*}
  \matrix[diagram,column sep={6em,between origins}]
  {
    |(C)| \CompCat_{\Id,\tSigma} &[1em] |(T)| \Trb & |(F)| \FibCat & |(Q)| \Lex_\infty \\
  };

  \draw[->] (C) to (T);
  \draw[->] (T) to (F);
  \draw[->] (F) to (Q);
  \end{tikzeq*}
  where the first functor is an equivalence by \Cref{compcat-trb},
  the second one by \Cref{trb-fibcat-DK} and
  the last one by \cite{s3}*{Thm.\ 4.9}.
\end{proof}